%% file: sec_00main.tex
\documentclass[lettersize,journal]{IEEEtran}
\usepackage{color,soul}
\usepackage{cite}
\usepackage{textcomp}
\usepackage{amsmath}
\usepackage{amssymb}
\usepackage{amsthm}
\usepackage[ruled,linesnumbered]{algorithm2e}
\newtheorem{thm}{Theorem}
\newtheorem{lem}{Lemma}
\newtheorem{cor}{Corollary}
\newtheorem{defn}{Problem}

\usepackage{xcolor}
\definecolor{darkgreen}{rgb}{0,0.5,0}

\usepackage{tikz}
\usetikzlibrary{automata, positioning, arrows}
\newtheorem*{remark}{Remark}

\usepackage{amsmath,amsfonts}
\usepackage{array}
\usepackage[caption=false,font=normalsize,labelfont=sf,textfont=sf]{subfig}
\usepackage{stfloats}
\usepackage{verbatim}
\usepackage{graphicx}
\usepackage{nomencl}
\makenomenclature
\usepackage{etoolbox}
\renewcommand\nomgroup[1]{%
  \item[\bfseries
  \ifstrequal{#1}{A}{Parameters of Problem (12)}{%
  \ifstrequal{#1}{B}{Variables of Problem (12)}{%
  \ifstrequal{#1}{C}{Parameters of Algorithm 1}{}}}%
]}

\begin{document}
\title{Data-Driven Distributionally Robust Electric Vehicle Balancing for Autonomous Mobility-on-Demand Systems under Demand and Supply Uncertainties}
\author{
{Sihong He,} \and {Zhili Zhang,} \and{Shuo Han,} \and {Lynn Pepin,} \and {Guang Wang,} \and {Desheng Zhang,} \and {John Stankovic,} \and {Fei Miao}
\thanks{This work is partially supported by NSF S\&AS-1849246, NSF CPS-1932250 and NSF S\&CC-1952096. Sihong~He, Zhili~Zhang, Lynn~Pepin, and Fei~Miao are with the Department of Computer Science and Engineering, University of Connecticut, Storrs Mansfield, CT, USA 06268.  Email: \{sihong.he, zhili.zhang, lynn.pepin, fei.miao\}@uconn.edu. This work is also partially supported by NSF S\&AS-1849238 and CPS-1932223. Shuo Han is with the Department of Electrical and Computer Engineering, University of Illinois, Chicago, IL, USA 60607. Email:hanshuo@uic.edu. Guang Wang is with the Department of Computer Science, Florida State University, Tallahassee, FL 32304. Email: guang@cs.fsu.edu. Desheng Zhang is with the Department of Computer Science, Rutgers University, Piscataway, NJ, USA 08901. Email: desheng.zhang@cs.rutgers.edu. John A. Stankovic is with the University of Virginia, Charlottesville, VA, USA 22904. Email: stankovic@cs.virginia.edu. 
}
}

\maketitle

\begin{abstract}
Electric vehicles (EVs) are being rapidly adopted due to their economic and societal benefits. Autonomous mobility-on-demand (AMoD) systems also embrace this trend. However, the long charging time and high recharging frequency of EVs pose challenges to efficiently managing EV AMoD systems. The complicated dynamic charging and mobility process of EV AMoD systems makes the demand and supply uncertainties significant when designing vehicle balancing algorithms. In this work, we design a data-driven distributionally robust optimization (DRO) approach to balance EVs for both the mobility service and the charging process. The optimization goal is to minimize the worst-case expected cost under both passenger mobility demand uncertainties and EV supply uncertainties. We then propose a novel distributional uncertainty sets construction algorithm that guarantees the produced parameters are contained in desired confidence regions with a given probability. To solve the proposed DRO AMoD EV balancing problem, we derive an equivalent computationally tractable convex optimization problem. Based on real-world EV data of a taxi system, we show that with our solution the average total balancing cost is reduced by 14.49\%, and the average mobility fairness and charging fairness are improved by 15.78\% and 34.51\%, respectively, compared to solutions that do not consider uncertainties.
\end{abstract}


    \nomenclature[A, 01]{\(K\)}{ number of time intervals in a day}
    \nomenclature[A, 02]{\(N\)}{ number of regions}
    \nomenclature[A, 03]{\(\tau\)}{model predicting time horizon}
    \nomenclature[A, 04]{$r \in \mathbb{R}^{N \tau}$}{ concatenated demand vector with unknown distribution function $F^*_r$}
    \nomenclature[A, 05]{$c\in \mathbb{R}^{N\tau}$}{ concatenated newly supply vector with unknown distribution function $F^*_c$}
    \nomenclature[A, 06]{$W \in \mathbb{R}^{N \times N}$}{cost matrix for vacant EVs, $w_{ij}$ is the moving cost from region $i$ to $j$ for one vacant EV }
    \nomenclature[A, 07]{$W^* \in \mathbb{R}^{N \times N}$}{cost matrix for low-battery EVs, $w^*_{ij}$ is the moving cost from region $i$ to $j$ for one low-battery EV }
    \nomenclature[A, 08]{$P^k_v, P_o^k,P_l^k, Q^k_v, Q^k_o$}{region transition matrices from time $k$ to $(k+1)$}
    \nomenclature[A, 09]{$V^1\in \mathbb{N}^{N}$}{the initial number of vacant EVs at each region provided by streaming data}
    \nomenclature[A, 10]{$O^1 \in \mathbb{N}^{N}$}{the initial number of occupied EVs at each region provided by streaming data}
    \nomenclature[A, 11]{$m_1,m_2 \in \mathbb{R}_+$}{ upper bound of distance each EV can drive idly}
    \nomenclature[A, 12]{$a \in \mathbb{R}_+$}{ power on the denominator of the objective}
    \nomenclature[A, 13]{$\theta,\beta \in \mathbb{R}_+$}{ weight factors of the objective function}
    \nomenclature[A, 14]{$l,h \in \mathbb{R}^{\tau}$}{lower and upper bounds of mobility supply-demand ratio}
    
    \nomenclature[B, 01]{$X^k \in \mathbb{R}^{N \times N}_+$}{$x_{ij}^k$ is the number of vacant EVs dispatched from region $i$ to region $j$ at time $k$}
    \nomenclature[B, 02]{$Y^k \in \mathbb{R}^{N \times N}_+$}{$y_{ij}^k$ is the number of low-battery EVs dispatched from region $i$ to region $j$ at time $k$}   
	\nomenclature[B, 03]{$V^k \in \mathbb{R}^{N}_+$}{ number of vacant EVs at each region before dispatching at the beginning of time $k$}
	\nomenclature[B, 04]{$O^k \in \mathbb{R}^{N}_+$}{ number of occupied EVs at each region before dispatching at the beginning of time $k$}
	\nomenclature[B, 05]{$L^k \in \mathbb{R}^{N}_+$}{ number of low-battery EVs at each region before dispatching at the beginning of time $k$}
	\nomenclature[B, 06]{$S^k \in \mathbb{R}^{N}_+$}{ number of vacant (available) EVs at each region after dispatching at time $k$}
	\nomenclature[B, 07]{$D,U \in \mathbb{R}^{N\tau}_+$}{ slack variables in constraint (8)}

	\printnomenclature[0.8in] \label{NOMENCLATURE}
	
\input{sec_01intro}

\input{sec_02related_work}

\input{sec_03prob_formulation}

\input{sec_04algorithm}

\input{sec_05theory}

\input{sec_06experiments}

\input{sec_07conclusion}

\input{sec_08appendix}

\bibliographystyle{ieeetr} 
{ \small 
\bibliography{sec_10ref}
}

\end{document}

%% file: sec_01intro.tex
\section{Introduction}
\label{intro}
In Autonomous Mobility-on-Demand (AMoD) systems, self-driving vehicles provide personal on-demand transportation service for customers and rebalance themselves to maintain acceptable quality of service throughout the system~\cite{zhang2016model,zhang2016control}. \textcolor{black}{AMoD systems have been advocated as one of the most promising energy-efficient transportation solution. } Electric Vehicles (EVs) have tremendous potential in AMoD systems for being economical and environmentally friendly~\cite{mitchell2010reinventing}. \textcolor{black}{For instance, EV AMoD systems directly address the problems of oil dependency and air pollution.} However, EVs have quite different energy-refilling patterns compared with traditional gas-powered vehicles. They have long charging times, high charging frequency, uncertain sporadic demands, and dispersed mobility patterns~\cite{wang2019experience, ammous2018optimal}. 

\textcolor{black}{There are emerging problems when commercial EVs are gradually introduced into AMoD systems, considering current charging technologies and limited charging infrastructures.} EVs' frequent intermittent charging requirements may reduce the quality of mobility service. Unbalanced EV distribution may also cause long waiting times and low charging service quality in some charging stations. To address these challenges, researchers proposed a plenty of vehicle allocation methods, charging scheduling approaches and joint charging-relocation recommendation schemes. \textcolor{black}{Vehicle allocation methods rebalance the vehicle distributions over time in responding to known or predicted demand and supply~\cite{hao2020robust, chavhan2020novel, zhang2016model}. We further discuss these methods in the Related Work section}. 

However, most existing vehicle allocation methods and charging scheduling approaches do not consider uncertainties by assuming the measurement and prediction models are perfect~\cite{iglesias2019bcmp, zhang2016model}, while model uncertainty affects the performance of decisions~\cite{zhen2021mathematical}. And it is difficult to accurately predict passengers' demand and EVs' charging patterns. As the promotion of EVs continues, we cannot ignore the uncertainties caused by EVs' charging behaviors. For example, uncertain charging time, queuing time, charging frequency all contribute to the supply uncertainty, due to limited knowledge we have about charging patterns~\cite{wang2019experience}. \textcolor{black}{Then difficulties appear in many aspects, such as introducing multiple uncertainties into AMoD systems, modeling the uncertain parameters and analyzing the mutual dependency between supply and demand.} Therefore, making real-time decisions under both supply and demand uncertainties are still charging and unsolved research problems. 

Considering both passenger demand and EV supply uncertainties, we propose a distributionally robust optimization (DRO) approach to make robust vehicle balancing decisions for both mobility service and charging scheduling. \textcolor{black}{DRO method considers the uncertain parameters' probability distributions are contained in some pre-specified distributional uncertainty sets~\cite{zhen2021mathematical}. More DRO literature are discussed in the Related Work section.} \textcolor{black}{In our proposed method, we assume the true probability distribution of passenger demand and EV supply lies in a set of probability distributions, i.e., a distributional uncertainty set.} We define the vehicle balancing cost and system-level charging service fairness requirement in our objective function, which is convex over the decision variables and concave over the uncertain parameters. The vehicle balancing cost includes the balancing cost to send vacant EVs towards predicted mobility demand and low battery EVs to charging stations. We put the mobility service fairness requirement in convex constraints. This objective and constraint design decouples the mutual dependencies between EV supply and passenger demand. \textcolor{black}{Thus, we calculate balancing decisions by solving the DRO problem, i.e. minimizing the worst-case objective function over distribuional uncertainty sets and convex constraints.} We further derive an equivalent convex optimization problem form for the DRO problem to provide solutions in a computationally tractable way. We also propose efficient distributional uncertainty set construction algorithms to construct stable uncertainty set. \textcolor{black}{We briefly summarize the proposed framework structure of this paper in Fig.~\ref{fig_summary}.}

The \textit{key contributions} of our work are as follows:

\begin{itemize}
 	\item To the best of our knowledge, our proposed mathematical system-level vehicle balancing framework is the first to consider both future mobility demand uncertainties and EV supply uncertainties for EV AMoD systems. While model predictive control algorithms~\cite{zhang2016model, iglesias2018data, miao2017data_uncertainty} have been designed considering AMoD system demand uncertainties in the literature, the supply side uncertainties for EV AMoD are not well studied yet.
 	\item We design a distributionally robust optimization approach to balance EVs across a city to provide fair passenger mobility and EV charging service while reducing the total balancing cost. We consider probabilistic distribution uncertainties of both the passenger mobility demand and the EV supply caused by the challenge of charging process prediction~\cite{ammous2018optimal,tian2016real}. The proposed problem formulation decouples the mutual dependencies between EV supply and passenger demand. \textcolor{black}{We further design an efficient algorithm to construct distributional uncertainty sets, and prove that the produced uncertainty set parameters are guaranteed being contained in desired  confidence regions with a given probability.}
 	\item We derive an equivalent form of a convex optimization problem for the proposed distributionally robust optimization problem. Hence, we provide a system-level performance guarantee in a computationally tractable way under supply and demand uncertainties. Based on EV taxi fleet of Shenzhen city, which is a real-world EV AMoD system dataset, we show that our method reduces the average total  balancing cost by 14.49\%, the average mobility unfairness and charging unfairness by  15.78\% and 34.51\%, respectively, compared to non-robust solutions.  
 \end{itemize}
 
The rest of the paper is organized as follows. The related work and distributionally robust EV balancing problem formulation are presented in Section~\ref{sec:realted_work} and~\ref{sec:problem_formulation}, respectively. The formal distribuitonally uncertainty set form and novel construction algorithms are in Section~\ref{sec:algorithm}. An equivalent computationally tractable form is derived in Section~\ref{sec:theory}. Experiments are in Section~\ref{sec:experiment}. We conclude in Section~\ref{sec:conclusion}.
\vspace*{-5pt}

%% file: sec_02related_work.tex
\begin{figure}[h]
	\centering
	\includegraphics [width=0.4\textwidth]{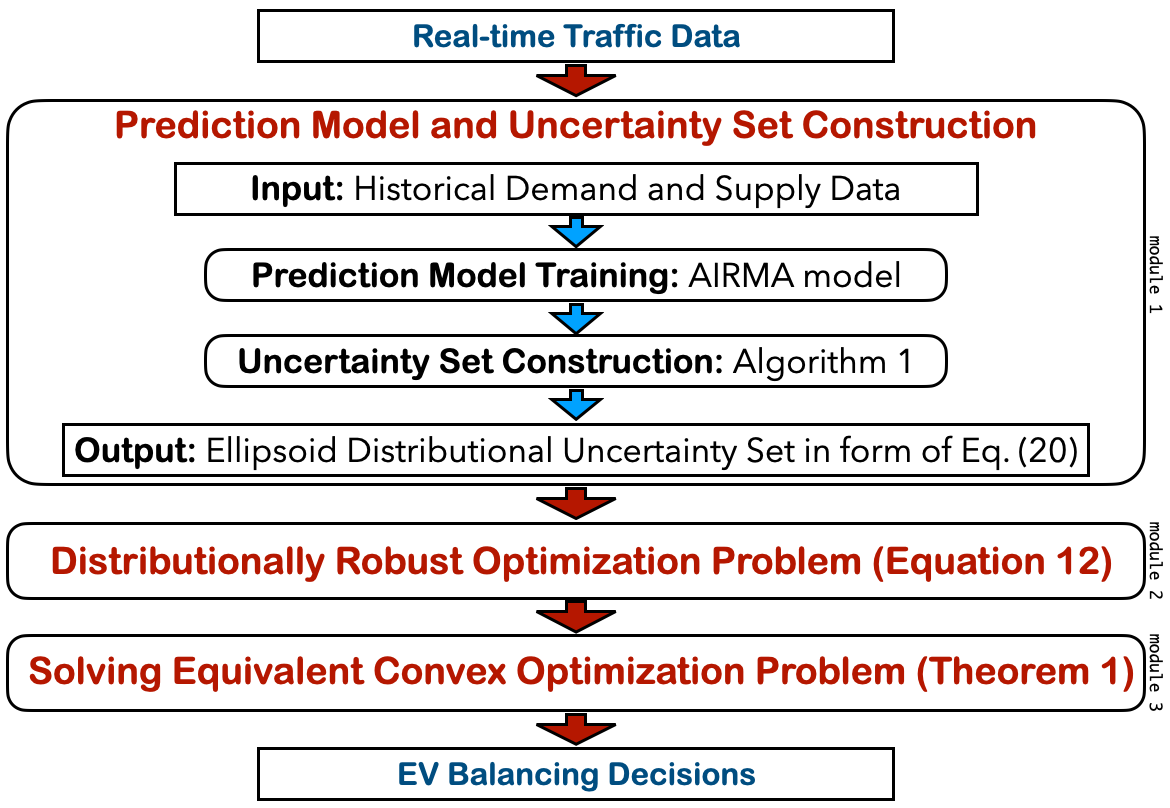}
	\vspace{-10pt}
	 
	\caption{In module 1, we use historical data to train a prediction model for supply and demand then use the well-trained prediction models as the input of our Algorithm 1 to finally get the distributional uncertainty sets for demand and supply in form of \eqref{set:uncertainty set}. In module 2, we get the DRO EV balancing problem (the distributionally robust optimization problem \eqref{opt_final}) using the real-time sensing data and distributional uncertainty sets. In module 3, we apply the conclusions in Theorem \ref{thm:theory1} to obtain the EV balancing decisions by solving the equivalent convex optimization problem \eqref{obj_theory}.}
	\color{black}
	\label{fig_summary}
	\vspace{-15pt}
\end{figure}
\section{Related Work}
\label{sec:realted_work}

\textbf{A. Vehicle Allocation: }To improve the performance of AMoD systems, multiple vehicle allocation and balancing approaches have been proposed. \textcolor{black}{For instance, queuing network model~\cite{iglesias2019bcmp}, flow framework~\cite{wollenstein2021routing},  model predictive control~\cite{zhang2016model, iglesias2018data}, receding horizon control~\cite{miao2016taxi}, and reinforcement learning method~\cite{liu2020context, he2022robust} have been designed}. \textcolor{black}{However, most of them do not consider EV charging patterns nor uncertainties caused by EV charging behaviours. Making real-time decisions under supply and demand uncertainties is still a challenging and unsolved problem. Our work jointly considers the EV allocation and charging problem in a distributionally robust optimization problem while considering two-side uncertainties. }

\textbf{B. Charging Scheduling: } \textcolor{black}{To improve EV charging process efficiency, MDP based~\cite{lin2021toward} and queuing model based~\cite{ammous2018optimal} charging scheduling, charging station deployment~\cite{faridimehr2018stochastic}, online charging recommendations~\cite{tian2016real} have been proposed.} Future charging supply or demand are usually considered in charging recommendation~\cite{yuan2021sac,abdullah2021reinforcement}. But these methods haven't provided integrated passenger picking-up and EV charging scheduling solutions for EV AMoD systems. \textcolor{black}{Though an E-taxi charging framework under dynamics of renewable energy and passenger mobility is proposed~\cite{yuan2021sac, estandia2021interaction}, its performance may be undermined by system uncertainties that are not considered. Other existing work either only focus on robust charging scheduling with uncertainties from charging behavior~\cite{zhao2017robust, yang2015noncooperative} or only study joint EV balancing and charging strategies without considering uncertainties~\cite{yuan2021sac}.} It is still challenging to simultaneously deal with passenger mobility demand and EV supply uncertainties when making EV balancing decisions for both passengers picking-up and EV charging. \textcolor{black}{Our work fills this gap by proposing a novel distributionally robust optimization EV balancing framework to provide fair passenger mobility and EV charging service while reducing the total balancing cost.}

\textbf{C. Robust and Distributionally Robust Optimization: }Robust optimization (RO) assumes that uncertain parameters can be any value in an uncertainty set, whereas distributionally robust optimization (DRO) models the uncertain parameters as random variables whose underlying probability distribution is contained in a distributional uncertainty set~\cite{zhen2021mathematical}. In both case, the goal is to find the best decision in view of the worst-case realization of uncertainty. However, RO may propose overly conservative decisions than DRO since they do not exploit distributional information~\cite{wiesemann2014distributionally}. Both RO and DRO have wide application in many disciplines, such as energy, healthcare, transportation, logistics and inventory, etc \cite{delage2010distributionally,wiesemann2014distributionally,gabrel2014recent,lin2022distributionally}. Set-membership methods use a deterministic unknown-but-bounded description of noise and parametric uncertainties \cite{wang2018set, pourasghar2016comparison}. Contrasted to such deterministic approaches, DRO is a stochastic approach that uncertain parameters are assumed to be follow some statistical distributions. \textcolor{black}{For AMoD system balancing, Hao et al. consider the idle vehicle pre-allocation problem with uncertain demands and covariate information using DRO~\cite{hao2020robust}; Miao et al. develop a data-driven DRO vehicle balancing method to accommodate uncertainties in the predicted demand distribution~\cite{miao2021data}. These methods only consider the demand uncertainty and cannot be directly applied to solve the challenge of integrally considering the demand and EV supply uncertainties.} To the best of our knowledge, we are the first to consider both future mobility demand distribution uncertainties and EV supply distribution uncertainties.

\vspace*{-12pt}

%% file: sec_03prob_formulation.tex
\section{Problem Formulation}
\label{sec:problem_formulation}
Mobility demand uncertainty has been considered in AMoD vehicle allocation or balancing~\cite{miao2021data, zhang2016model} and EV supply uncertainty has also been addressed in EV charging scheduling~\cite{zhao2017robust, yang2015noncooperative}. However, it is still challenging and has not been studied to consider these two-sided uncertainties simultaneously for AMoD systems using EV, where the EV balancing for mobility and charging process is tightly integrated and should be optimized jointly. 

In this section, we formulate the EV balancing problem as a distributionally robust optimization (DRO) problem considering both predicted passenger mobility demand and EV supply uncertainties. The DRO decision minimizes the worst-case expected cost over a set of uncertain supply-and-demand's probability distributions, and provides performance guarantee of the decisions under model uncertainties. We consider both passenger supply-demand ratio fairness and EV charging supply-demand ratio fairness. While previous work considers only passenger mobility supply-demand ratio fairness in the objective function~\cite{miao2021data}, we put it in the constraints and put EV charging supply-demand ratio fairness in the objective function (see Eq.~\eqref{opt_final} for detail). Such a problem formulation makes it possible to consider both charging and mobility fairness under two-sided uncertainties, and decouples the mutual dependencies between EV supply and passenger demand under complex dynamics between EV charging and mobility patterns.

The balancing decisions are updated in a receding horizon control process~\cite{chen2020optimal, miao2016taxi}. At each time step, the dispatching center first updates vehicle status and passenger demand information, then calculates the EV balancing decisions, and finally sends the decisions to the EVs to execute. Vacant EVs are allocated among different regions to pick up current and predicted future passengers. Low-battery EVs are dispatched to regions with charging stations to charge. The dispatching center focuses on the global-level balancing among regions. A local controller finishes one-to-one or one-to-group (carpool) EV-passenger or EV-charging station matching and detailed routing. We focus on the system-level robust EV balancing method design within a city. Local-level trip assignment and routing algorithms are out of the scope of this work and are investigated in the literature~\cite{mourad2019survey, alonso2017demand, chen2017hierarchical}. Our system-level EV balancing method can be applied in conjunction with these local-level one-to-one or one-to-group matching algorithms.

\begin{figure}[ht] 
\vspace*{-8pt}
    \centering 
    \begin{tikzpicture}[shorten >=1pt,node distance=3cm,on grid,auto]
  \tikzstyle{every state}=[fill={rgb:black,1;white,10}]

    \node[state]   (q_1)                    {L};
    \node[state] (q_2)  [right of=q_1]    {V};
    \node[state]           (q_3)  [right of=q_2]    {O};

    \path[->]
    (q_1) edge [loop above] node {}    (   )
          edge [bend left]  node {finish charging}    (q_2)
    (q_2) edge [bend left]  node {start order}    (q_3)
          edge [bend left]  node {lack of energy}    (q_1)
          edge [loop above] node {}    (   )
    (q_3) edge [bend left]  node {finish order}  (q_2)
          edge [loop above] node {}    (   );
\end{tikzpicture}
\vspace*{-10pt}
\caption{Status transition process: the status "vacant" is a bridge between status "low-battery" and "occupied".}
\vspace*{-5pt}
\label{fig_transit}
\end{figure}
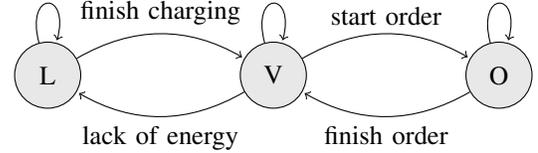

\vspace*{-15pt}
\subsection{EV Status and Corresponding Actions}
We define three status for one EV according to its battery level and working state: vacant, occupied and low-battery. A vacant EV: is one that is working and has a battery level is higher than a threshold $\textbf{e}$, but has no passengers. An occupied EV: is one that is working and has a battery higher than a threshold $\textbf{e}$ and has passengers. A low-battery EV: is one that is working and has no passengers in it and has a battery level lower than the threshold $\textbf{e}$. The controller dispatches vacant EVs according to current and predicted future passenger demands, and assigns low-battery EVs to regions where located charging stations. It has no actions for occupied EVs since these EVs are busy in serving passengers. An EV's status can transit among these three categories. Fig. \ref{fig_transit} shows the status transition process of the EVs according to the EV AMoD system dynamics \cite{miao2017data_uncertainty, miao2017data,chen2017hierarchical, zardini2022analysis, xu2018taxi}. A low-battery/occupied EV can only transfer to a vacant EV or stay in the current status. The status "vacant" is a bridge between the status "low-battery" and "occupied".
\vspace*{-10pt}
\subsection{DRO EV Balancing Problem for Mobility and Charging}
Our goal is to find robust EV balancing decisions when considering the randomness and prediction errors of both EV supply and passenger mobility demand. Hence, we formulate a distributionally robust optimization (DRO) problem to minimize the expected EV balancing cost and provide fair charging and mobility service over an uncertainty set of demand and supply probability distributions. Nomenclature subsection provides an overview for parameters and variables in the problem formulation and algorithm. 

We divide one day into $K$ time intervals and denote $k \in \{1,2,...,K\}$ as the temporal index. We separate a city into $N$ regions and denote $i \in \{1,2,...,N \}$ as the spatial index. We denote $r_i^k$ as the predicted total number of passengers demand, and $c_i^k$ as the predicted total number of vacant EVs that finish charging and turn to supply in region $i$ at time interval $k$. We then define $r^k = [r_1^k, r_2^k,...,r_N^k]^{\top} \in \mathbb{R}^N$ and $c^k =[c_1^k, c_2^k,...,c_N^k]^{\top} \in \mathbb{R}^N$ as vectors containing the demand and supply of each region in time interval $k$, $r=[r^1, r^2, ... , r^{\tau}] \in \mathbb{R}^{N\tau}$ and $c= [c^1, c^2, ... , c^{\tau}] \in \mathbb{R}^{N\tau}$ as the concatenation of demand and supply respectively. To consider model uncertainties, we assume they are random vectors instead of deterministic vectors. We use $F^*_r$ and $F^*_c$ to denote the unknown true probability distributions of $r,c \in \mathbb{R}^{N\tau}$ respectively, i.e., $r \sim F^*_r$ and $c \sim F^*_c$. We use non-negative matrices $X^k, Y^k \in \mathbb{R}^{N\times N}$ as the decision matrices, where $x_{ij}^k, y_{ij}^k$ is the total number of vacant, low-battery EVs that will be dispatched from region $i$ to region $j$ at the beginning of time interval $k$, respectively. Minimizing the expected vacant and low-battery EVs balancing cost under random demand vector $r$ and supply vector $c$ is described as the following stochastic programming (SP) problem:
\begin{align}
	\begin{split}
		\underset{X^{1:\tau}, Y^{1:\tau}}{\text{min.}}\ &\mathbb{E}_{r\sim F^*_r,c\sim F^*_c}\left[J(X^{1:\tau}, Y^{1:\tau},r,c) \right] \\
		\text{s.t.}\quad  & X^{1:\tau}, Y^{1:\tau}\in \mathcal{D},
	\end{split}
	\label{opt_stochastic}
\end{align}
where $J(X^{1:\tau}, Y^{1:\tau},r,c)$ is a cost function of allocating EVs given balancing decisions $X^{1:\tau} = \{X^1, X^2,\dots,X^{\tau}\}$ and $Y^{1:\tau} = \{Y^1, Y^2,\dots,Y^{\tau}\}$. $\mathcal{D}$ defines the convex constraints domain of the decision variables.

However, in the real world, we usually have limited knowledge about the true probability distributions $F^*_r$ and $F^*_c$. With historical or streaming data, we can estimate sets of probability distributions $\mathcal{F}_r$ and $\mathcal{F}_c$ such that $F^*_r \in \mathcal{F}_r$, $F^*_c \in \mathcal{F}_c$~\cite{tian2016real, wang2019sharedcharging, rossi2019interaction}, instead of the exact $F^*_r$ and $F^*_c$. Meanwhile, problem \eqref{opt_stochastic} is computationally expensive to solve in real-time~\cite{delage2010distributionally}. Hence, we consider minimizing the worst-case expected cost, i.e., a minimax form of problem \eqref{opt_stochastic}: 
\begin{align}
	\begin{split}
		\underset{X^{1:\tau},Y^{1:\tau}}{\text{min.}}\ \underset{F_r\in \mathcal{F}_r,F_c\in \mathcal{F}_c}{\text{max.}}\quad &\mathbb{E} \left[J(X^{1:\tau}, Y^{1:\tau},r,c) \right]\\
\text{s.t.}\quad  &X^{1:\tau}, Y^{1:\tau}\in \mathcal{D}.
	\end{split}
	\label{opt_minmax}
\end{align}
Problem \eqref{opt_minmax} is a distributionally robust optimization (DRO) problem~\cite{delage2010distributionally, zhen2021mathematical} that assumes $F^*_r \in \mathcal{F}_r$, $F^*_c \in \mathcal{F}_c$. In the following context, we define the complete forms of objective function and constraints. The formal definitions and construction algorithms of the distributional uncertainty sets $\mathcal{F}_r$ and $\mathcal{F}_c$ are introduced in Section~\ref{sec:algorithm}. 
\vspace*{-3pt}

\begin{remark}[Difficulties of considering multiple mutual effected parameter uncertainties]
DRO methods have been designed in the literature~\cite{miao2021data} for vehicle balancing problems for gasoline vehicle MoD system,  where only the passenger mobility demand uncertainty has been considered. In contrast, our work considers the uncertainties of multiple parameters, i.e. demand and supply uncertainties. It is not straightforward to apply DRO methods that only consider a single parameter uncertainty to solve problems with multiple uncertain parameters, especially in a complicated and dynamic transportation system. The uncertain parameters of the system model are usually dynamically coupled  \cite{noyan2018distributionally, luo2020distributionally}. This mutual-dependent property results in two-fold difficulties. One is that formulating a computationally tractable DRO problem for vehicle balancing of a EV MoD system considering both the charging and passenger service processes gets more challenging. Considering the intrinsic connections of uncertain parameters, the DRO problem formulation should be carefully designed to satisfy necessary conditions of computational tractability, which is even challenging with a single uncertain parameter as shown in the literature \cite{delage2010distributionally}. To keep the problem formulation convex over decision variables and concave over uncertain parameters, we define the system performance requirements such as charging fairness in the objective function as a fractional function, and the mobility service fairness in the constraint functions as linear inequalities. We define the mobility dynamic process and status transition process of EVs as linear constraints. We explain all the detail in the problem formulation section. Another challenge is that numerically solving a DRO problem like \eqref{opt_final} can be challenging given the definitions of the parameter uncertainty sets \cite{delage2010distributionally,noyan2018distributionally, luo2020distributionally}. In this work, we decouple the dependencies of uncertain parameters in the problem formulation \eqref{opt_final}, and then derive a computationally tractable and equivalent convex optimization problem for \eqref{opt_final} in Section \ref{sec:theory}.
\end{remark}
\color{black}

\vspace*{-15pt}

\subsection{Balancing Cost}

One optimization goal is to minimize the worst-case EV balancing cost. We define $W \in \mathbb{R}^{N \times N}$ as the cost matrix where $w_{ij}$ is the cost of sending a vacant EV from region $i$ to region $j$. The cost can be the approximated distance given a specific region partition method, approximated routing distance or travel time between two regions, etc. We select the approximated routing distance as the cost definition in our data-driven experiments. We also define $W^* \in \mathbb{R}^{N \times N}, w^*_{ij}$ as the cost matrix of sending a low-battery EV. We set $w^*_{ij} = \infty$ if there are no charging stations in region $j$, since low-battery EV should not go to regions without charging stations. 

Then the total balancing cost function $J_D$ for $\tau$ intervals is defined as~\eqref{obj_dist}, where $\beta$ is a positive weight coefficient. 
\begin{align}
	J_D (X^{1:\tau}, Y^{1:\tau}) := \sum\limits_{k=1}^{\tau} \sum\limits_{i=1}^{N} \sum\limits_{j=1}^{N} (x^k_{ij} w_{ij} + \beta y^k_{ij} w^*_{ij}).
	\label{obj_dist}
\end{align}
\vspace*{-20pt}
\subsection{Constraints Definition}

\subsubsection{Mobility Dynamics Constraints}

We define the mobility dynamics constraints in~\eqref{def_cons_mobility} which describes the EV status transition. Let $V_i^k, O_i^k \in \mathbb{R}$ be the number of vacant and occupied EVs in region $i$ at the beginning of time $k$ before balancing, respectively. We define $S_i^k$ as the total number of vacant EVs that are available to serve in region $i$ after executing balancing decisions at time $k$. Then the following Equations~\eqref{def_cons_mobility} of $V_i^k, O_i^k, S_i^k$ describe the EV AMoD dynamics \cite{miao2017data_uncertainty, miao2017data,chen2017hierarchical, zardini2022analysis, xu2018taxi} for $k \in \{1,\dots,\tau-1\}$: 
\vspace*{-5pt}

\begin{align}
	\begin{split}
		 S^k_i &=\sum\limits_{j=1}^{N} x^k_{ji}-\sum\limits_{j=1}^{N} x^k_{ij} + V^k_i, k \in \{1,\dots,\tau\}; \\
		V^{k+1}_i&= \sum\limits_{j=1}^{N} P^k_{vji}S^k_j+ \sum\limits_{j=1}^{N} Q^k_{vji}O^k_j + c_i^k,\\ 
		 O^{k+1}_i&=\sum\limits_{j=1}^{N}P^k_{oji}S^k_j+ \sum\limits_{j=1}^{N} Q^k_{oji}O^k_j, k \in \{1,\dots,\tau-1\};
	\end{split}
	\label{def_cons_mobility}
\end{align}
where $P^k_v, P^k_o, Q^k_v, Q^k_o \in \mathbb{R}^{N \times N}$ are region transition matrices: $P^k_{vji}(P^k_{oji})$ is the probability that a vacant EV moves from region $j$ at the beginning of time $k$ will transverse to region $i$ and being vacant (occupied) at the beginning of time $k+1$. Similarly, $Q^k_{vji}(Q^k_{oji})$ is the probability that an occupied EV moves from region $j$ at time $k$ will go to region $i$ and being vacant (occupied) at the beginning of time $k+1$. The mobility constraints \eqref{def_cons_mobility} show that (i) the number of vacant EVs after balancing is related to the the number of vacant EVs before balancing and the net change number of vacant EVs according to the balancing decisions; (ii) there are two sources of occupied EVs: one is former occupied EVs that are still occupied, and the other is former vacant EVs that change into occupied; (iii) similarly, there are two sources of low-battery EVs: former low-battery EVs that stay in low-battery status, and former vacant EVs that change into low-battery status. The method of calculating region transition matrices is introduced in the literature~\cite{miao2016taxi}. When receding the time horizon, locations and status of all EVs are updated by real-time sensing data therefore $V^1, O^1$ are provided realtimely.

\subsubsection{Charging Dynamics Constraints}
We define $L_i^k$ as the total number of low-battery EVs in region $i$ before balancing at the beginning of time $k$, the charging dynamics constraints in Equation~\eqref{def_cons_charging} states the quantitative relationship between low-battery EVs and available vacant EVs. 
\vspace*{-5pt}
\begin{align}
		L^{k+1}_i &=\sum\limits_{j=1}^{N} y^k_{ji}-\sum\limits_{j=1}^{N} y^k_{ij} + \sum\limits_{j=1}^{N} P_{lji}^k S^k_{j},
	\label{def_cons_charging}
\end{align}
where $L^1$ is given by real-time data, $P^k_l \in \mathbb{R}^{N \times N}$ is the region transition matrix. Here $P^k_{lji}$ is the probability that a vacant EV moves from region $j$ at the beginning of time $k$ will go to region $i$ and being low-battery at the beginning of time $k+1$. The region transition matrices estimated from data satisfy that $\sum\limits_{j=1}^{N} P^k_{lij} + P^k_{vij} + P^k_{oij} = 1 \ \text{and} \; \sum\limits_{j=1}^{N} Q^k_{oij} + Q^k_{vij} = 1.$ The method of calculating region transition matrices is introduced in the literature~\cite{miao2016taxi}.

\subsubsection{Moving Constraints}
We also have moving constraints for decision variables $X^k$ and $Y^k$ defined in~\eqref{def_cons_moving}. When the balancing cost $w_{ij}$ stands for distance, the idle driving distance EVs can move during a given time interval is limited, either due to speed limit or insufficient battery. Hence, 
\begin{align}
	\begin{split}
		x_{ij}^k \geq 0 \text{ and } x_{ij}^k = 0 \text{ when } w_{ij} \geq m_1;\\
        y_{ij}^k \geq 0 \text{ and } y_{ij}^k = 0 \text{ when } w^*_{ij} \geq m_2,
	\end{split}
	\label{def_cons_moving}
\end{align}
where $m_1 > 0$ and $m_2 > 0$ is the upper bound moving distance for one vacant EV, and one low-battery EV, respectively. The values of $m_1$ and $m_2$ can be obtained by applying the comprehensive investigation results and methods in mobility and charging patterns of EVs from the literature~\cite{wang2019experience, wang2020faircharge}.
\vspace*{-10pt}
\subsection{Mobility Supply-Demand Ratio}
Mobility supply-demand ratio fairness is one service quality metric for AMoD systems~\cite{zhang2016model, miao2021data, liu2020context}. In the literature, this goal is usually defined as minimizing the total absolute difference between local and global mobility supply-demand ratio for $\tau$ time intervals, i.e., $J_M(X^{1:\tau})$ defined in \eqref{def_mobility_unfairness}.
\begin{align}
J_M(X^{1:\tau}) :=	\sum_{k=1}^{\tau}\sum_{i=1}^{N}\left| \frac{r^k_i}{S^k_i} - \frac{\sum_{j=1}^{N}r^k_j}{\sum_{j=1}^{N} S^k_j}\right|.
	\label{def_mobility_unfairness}
\end{align}
Note that the uncertain parameter $c^k_i$ is related to $S^k_i$ as defined in \eqref{def_cons_mobility}, the uncertain parameters $r^k_i$ and $c^k_i$ are included in the numerator and denominator of mobility demand-supply ratio, respectively. Hence, directly minimizing $J_M(X^{1:\tau})$ is computationally intractable. Instead of minimizing $J_M(X^{1:\tau})$ in~\eqref{def_mobility_unfairness}, we consider the following mobility fairness constraints to make sure the mobility supply-demand ratio of each region is within a same range to provide fair service: $l^k \leq {r_i^k}/{S_i^k} \leq h^k, \quad k \in \{1,\dots,\tau\}$,
where $l^k(h^k)$ is the lower (upper) bound of the mobility supply-demand ratio at time $k$. The value of $l^k$ and $h^k$ can be decided by historical data. We transfer those inequalities to the following Equations form with slack variables $D_i^k, U_i^k$:  
\begin{align}
	\begin{split}
		&r_i^k - l^k S_i^k - (D_i^k)^2 = 0,\ 
        r_i^k - h^k S_i^k + (U_i^k)^2 = 0,
	\end{split}
	\label{def_cons_mobility_unfairness}
\end{align}
\vspace*{-15pt}
\subsection{Charging Supply-Demand Ratio}
When a fixed number and locations of charging stations are given, to avoid long waiting time or queues at some charging stations, we send low-battery EVs to regions with charging stations according to the dynamic availability of charging spots at different regions.  Hence, we balance the charging supply-demand ratio for EV charging across the whole city~\cite{green2011queueing}. 

We denote $T_i^k$ as the net number of low-battery EVs in region $i$ after the low-battery EV balancing decision $Y^k$,
\begin{align}
 \label{def_net_low}
    T_i^k = \sum\limits_{j=1}^{N} y^k_{ji}-\sum\limits_{j=1}^{N} y^k_{ij}, \ \forall \ i \in \sigma.
\end{align}
where $\sigma$ is the set of regions with charging stations. And we set $T_i^k = 0, \ \forall \ i \not\in \sigma$, since no low-battery EVs should be dispatched to regions where located a charging station. The charging supply-demand ratio in region $i$ during time $k$ is approximated as the ratio of the new charging spots supply and the EV charging demand rate ${c_i^k}/{T_i^k}$, where $c^k_i$ is the number of EV that finished charging and become vacant vehicle supply in region $i$ during time $k$. When one EV finishes charging, there will be one newly available charging spot; hence, we use $c^k_i$ to approximate the amount of charging supply to serve low-battery EVs. \textcolor{black}{To balance the charging supply-demand ratio over the city, we minimize the total absolute difference between the local and global charging supply-demand ratio for $\tau$ time intervals, i.e. $J_1$ defined in \eqref{def_charging_sdr_unfairness}. Smaller $J_1(Y^{1:\tau})$ means the charging decisions $Y^{1:\tau}$ achieve higher charging fairness. Such fairness metrics have been widely used in literature \cite{miao2021data,miao2016taxi}.}
\begin{align}
J_1(Y^{1:\tau})=\sum_{k=1}^{\tau}\sum_{i \in \sigma}\left| \frac{c^k_i}{T^k_i} - \frac{\sum_{j \in \sigma}c^k_j}{\sum_{j \in \sigma} T^k_i}\right|.
	\label{def_charging_sdr_unfairness}
\end{align}
However, function~\eqref{def_charging_sdr_unfairness} is not concave over uncertainty parameter $c^k_i$. For computationally tractability, we define the objective of fair charging as to minimize the function $J_E$
\begin{align}
	\begin{split}
		J_E(Y^{1:\tau}) := \sum\limits_{k=1}^{\tau} \sum_{i \in \sigma} \frac{c_i^k}{(T_i^k)^a}.
	\end{split}
	\label{def_charging_unfairness}
\end{align}
\textcolor{black}{$J_E(Y^{1:\tau})$ has good properties that it is linear in $c^i_k$ and convex over $Y^{1:\tau}$ when the power parameter $a>0$. What's more, the function~\eqref{def_charging_unfairness} approximates function~\eqref{def_charging_sdr_unfairness} when $a$ is designed to be small enough, according to Lemma 1 in the literature~\cite{miao2021data}}.
\vspace*{-10pt}
\subsection{DRO EV Balancing Problem Formulation}
 With constraints \eqref{def_cons_mobility}, \eqref{def_cons_charging}, \eqref{def_cons_moving}, \eqref{def_cons_mobility_unfairness}, we finally define the distributionally robust EV balancing for mobility and charging process under uncertain probability distributions of the demand and supply in \eqref{opt_final}, where the objective function $J$ is a weighted sum of the balancing cost function $J_D$ defined in \eqref{obj_dist} and charging unfairness function $J_E$ defined in \eqref{def_charging_unfairness}.
\begin{align}
	\begin{split} 
	\underset{\substack{X^{1:\tau}, Y^{1:\tau}, S^{1:\tau}, D^{1:\tau},\\ U^{1:\tau},  V^{2:\tau}, O^{2:\tau}, L^{2:\tau}}}{\text{min.}} \
		 &\underset{\{F_r\in \mathcal{F}_r,F_c\in \mathcal{F}_c\} }{\text{max.}}\ \mathbb{E} \left[J\right]\\
		\text{s.t.}\ \text{\eqref{def_cons_mobility};}&\ \text{ \eqref{def_cons_charging};}\ \text{ \eqref{def_cons_moving};}\ \text{ \eqref{def_cons_mobility_unfairness},} 
	\end{split}
	\label{opt_final}
\end{align}
where $J$ is the final objective defined in \eqref{def_final_obj}, 
\begin{align}
    J = J_D(X^{1:\tau}, Y^{1:\tau})+\theta J_E(Y^{1:\tau}, c)
    \label{def_final_obj}
\end{align}
$\theta$ is a positive weighted parameter for a trade-off between $J_D$ and $J_E$. And $T_i^k = \mathcal{L}_T(Y^k)$ is a linear function of decision variables $Y^k$, $S_i^k = \mathcal{L}_S(X^{1:k}, c_{i}^{1:k-1})$ is a linear function of decision variables $X^{1:k} := \{ X^1, ..., X^k\}$ and supply uncertain parameter $c^{1:k-1}_i := \{c_i^1, ..., c_i^{k-1}\}$. We prove the following Lemma~\ref{lem_exist_bound}, there must exist a set of lower and upper bounds of the mobility supply-demand ratio in constraint~\eqref{def_cons_mobility_unfairness} to guarantee that there is a feasible solution for the proposed DRO problem~\eqref{opt_final}. 

\begin{lem}
\label{lem_exist_bound}
We can find at least one lower and one upper bound of the mobility supply-demand ratio such that the DRO problem~\eqref{opt_final} has at least one feasible solution. We name these bounds global bounds. And there exists a pair of lower and upper bounds $l^k$ and $h^k$ that no looser than global bounds as well as guarantee the DRO problem~\eqref{opt_final} has at least one feasible solution.
\end{lem}
\begin{proof}
See Appendix~\ref{proof_exist_bound}.
\end{proof}
\vspace*{-10pt}
\begin{remark}[Novelty of problem formulation]
The DRO EV balancing problem formulation~\eqref{opt_final} decouples the mutual dependencies between EV supply and passenger demand. By putting the mobility supply-demand ratio fairness requirement in the constraints and the EV charging supply-demand ratio fairness in the objective, we make it possible to consider both charging and mobility fairness under two-sided uncertainties. We then derive a computationally tractable convex optimization form of problem~\eqref{opt_final} in Section~\ref{sec:theory}.
\end{remark}

\vspace*{-15pt}
\subsection{Generalization for Heterogeneous EV Fleet}

\input{sup_07general_problem_form}

\begin{lem}
\label{lem_convex}
The generalized heterogeneous EV balancing problem~\eqref{ge_minmax}
\begin{align}
	\begin{split}
		\underset{\substack{X^{1:\tau}, Y^{1:\tau}, S^{1:\tau}, D^{1:\tau},\\ U^{1:\tau},  V^{2:\tau}, O^{2:\tau}, L^{2:\tau}}}{\text{min.}} \
		 &\underset{\{F_r\in \mathcal{F}_r,F_c\in \mathcal{F}_c\} }{\text{max.}}\mathbb{E} \left[ J_D^{\prime} + \theta J_E^{\prime} \right]\\
\text{s.t.}\quad  &\eqref{def_cons_mobility_unfairness}; \eqref{ge_bound}; \eqref{ge_trans}; \eqref{ge_con_charing}.
	\end{split}
	\label{ge_minmax}
\end{align}
is still convex with respect to the decision variables, and concave on uncertain parameters, where $J_E^{\prime}$ and $J_D^{\prime}$ are defined in~\eqref{ge_quality_charging_one} and \eqref{ge_obj_dist} respectively.
\end{lem}
\begin{proof}
See Appendix~\ref{proof_lem_convex}.
\end{proof}

%% file: sup_07general_problem_form.tex
We generalize the DRO problem \eqref{opt_final} with homogeneous EVs (i.e. all vacant EVs have the same capacity or number of available seats) to a heterogeneous EV balancing problem formulation that considers EVs with different capacities. 

We denote the capacity of one EV as $C_{e}$, where $e$ (typ\textbf{e}) is the index of different types as $e = 1, \cdots, E$. $E$ is the total number of capacity types. For instance, $C_e$ is the number of seats for type $e$ EVs. We use $V_{e,i}^k$ to denote the number of vacant EVs with capacity $C_e$ in region $i$ at time $k$ before balancing. So the total number vacant EVs in region $i$ at time $k$ before balancing is $V_i^k := \sum_{e=1}^E V_{e,i}^k$. Similarly, we use $S^k_{e,i}, O^k_{e,i}, L^k_{e,i}$ to denote numbers of available, occupied, low-battery EVs with $C_e$ capacity, respectively. Then $x_{e,ij}^k$ is the number of type-$e$ vacant EVs that will be dispatched from region $i$ to region $j$ in time interval $k$. And $y_{e, ij}^k$ is the decision for type-$e$ low-battery EVs. Then the balancing cost constraints~\eqref{def_cons_moving} become:
\begin{align}
		x_{e,ij}^k \geq 0 \text{ and } x_{e,ij}^k = 0 \text{ when } w_{ij} \geq m_{e,1}; \nonumber\\
        y_{e,ij}^k \geq 0 \text{ and } y_{e,ij}^k = 0 \text{ when } w^*_{ij} \geq m_{e,2},
	\label{ge_bound}
\end{align}
because different types of EVs may have different moving constraints. For example, EVs with a smaller number of seats may use less energy or time to finish the same-length trip compared to EVs with a larger number of seats. Because they usually benefit from lighter cargo and metal frame, as well as less traffic constraints (due to height limit, size limit, lane limit, etc.). We can determine these constraints parameters by applying the comprehensive investigation results and methods in mobility and charging patterns of EVs from the literature~\cite{wang2019experience, wang2020faircharge}. And the mobility dynamic constraints~(4) become: for $e = 1, \cdots, E$,
\begin{align}
        S^k_{e,i} &= \sum\limits_{j=1}^{N}x^k_{e, ji}-\sum\limits_{j=1}^{N} x^k_{e, ij} + V_{e,i}^k,\quad k=1,\cdots,\tau; \nonumber\\
		\label{ge_trans} V^{k+1}_{e,i}&= \sum\limits_{j=1}^{N} P^k_{e,vji}S^k_{e,j}+ \sum\limits_{j=1}^{N} Q^k_{e,vji}O^{k}_{e,j} + c_{e,i}^k,\\ 
		 O^{k+1}_{e,i}&=\sum\limits_{j=1}^{N}P^k_{e,oji}S^k_{e,j}+ \sum\limits_{j=1}^{N} Q^k_{e,oji}O^k_{e,j}, k=1,\cdots,\tau-1,\nonumber
\end{align}
where $S^k_{e,i}$ is the total number of type-$e$ vacant EVs in region $i$ at time $k$, $c^k_{e,i}$ is the number of type-$e$ low-battery EVs in region $i$ that finish charging in time $k$. The charging dynamic constraint~\eqref{def_cons_charging} turns to:
\begin{align}
		L^{k+1}_{e,i} &=\sum\limits_{j=1}^{N} y^k_{e, ji}-\sum\limits_{j=1}^{N} y^k_{e, ij} + \sum\limits_{j=1}^{N} P_{e,lji}^k S^k_{e,j}. 
	\label{ge_con_charing}
\end{align}
The region transition matrices $P^k_{e, vij}, P^k_{e, oij}, P^k_{e, lij}$ denote the probability that an available type-$e$ EV moves from region $i$ at the beginning of time $k$ will transverse to region $j$ and being vacant, occupied, low-battery, respectively at the beginning of time $k+1$. Similarly, $Q^k_{e, vij}, Q^k_{e, oij}$ denote the probability that an occupied type-$e$ EV becomes vacant, occupied, respectively in the transition process.
The objective $J_E$ has a new form that:
\begin{align}
		J_E^{\prime} := \sum\limits_{k=1}^{\tau} \sum\limits_{i \in \sigma} \frac{\sum_{e=1}^Ec_{e,i}^k}{(\sum_{e=1}^E T_{e,i}^k)^a}. 
	\label{ge_quality_charging_one}
\end{align}
The objective $J_D$ also has a new form that:
\begin{align}
	J_D^{\prime} := \sum_{e=1}^{E}\sum\limits_{k=1}^{\tau} \sum\limits_{i=1}^{N} \sum\limits_{j=1}^{N} (x^k_{e,ij} w_{ij} + \beta y^k_{e, ij} w^*_{ij}).
	\label{ge_obj_dist}
\end{align}
The objective function $J_E^{\prime}$ defined in Equation~\eqref{ge_quality_charging_one} is linear (concave) over $c^k$, convex over $Y^{1:\tau}_{1:E}$. The objective function $J_D^{\prime}$ defined in Equation~\eqref{ge_obj_dist} is convex over $X^{1:\tau}_e$ and $Y^{1:\tau}_e$, $e = 1, \cdots, E$, since linear operation preserves convexity~[47]. The generalized balancing cost constraints~\eqref{ge_bound}, mobility dynamic constraints~\eqref{ge_trans} and charging dynamic constraints~\eqref{ge_con_charing} are linear of decision variables. So we have Lemma~2.

%% file: sec_04algorithm.tex
\section{Distributional Uncertainty Set Form and Construction Algorithm}
\label{sec:algorithm}

We design efficient algorithms to construct the demand and supply uncertainty sets $\mathcal{F}_r, \mathcal{F}_c$ of probability distributions defined in problem~\eqref{opt_final}, with historical data that contains information related to the true demand and supply distributions. The works in~\cite{delage2010distributionally} and~\cite{bertsimas2018data} use empirical estimates to construct the uncertainty set according to confidence regions of hypothesis testing in portfolio management problems. The method in~\cite{miao2021data} leverages the structure property of the covariance matrix to develop efficient uncertainty sets construction algorithms. However, the literature methods do not provide stable uncertainty set guarantee. Hence, we design a bootstrap-based algorithm that produces stable uncertainty set given different prediction models, and prove that the produced uncertainty set parameters are guaranteed to be contained in desired confidence regions with a given probability. 

\vspace*{-10pt}
\subsection{General Distributional Uncertainty Set}
Without loss of generality, we denote $z$ as a random vector variable and denote $\{ z_1, ..., z_n \}$ as a set of sample vectors; $f(o)$ as a prediction model of random variable $z$ and $\hat{z}$ as the prediction of $z$ by the model $f(o)$, i.e., $\hat{z} = f(o)$, where $o$ is the input data. We use $\delta$ to denote the prediction error which is the difference between the true value and the estimated value, i.e., $z = \hat{z} + \delta$. Therefore, $\delta$ is also a random vector variable. We define the following distributional uncertainty set construction problem for random variable $z$. We then calculate the uncertainty sets for both random demand $r$ and supply $c$ according to formulation~\eqref{set:uncertainty set}.
\begin{defn}
\label{def:problem}
Given a sample set of the random vector variable $z$, a prediction method $f(o)$, find values of $\hat{z}, \hat{\Sigma}$, $\hat{\omega}$ and $\hat{\gamma}$, such that with probability at least $1-\alpha$ the true distribution of $z$ is contained in the distributional uncertainty set~\eqref{set:uncertainty set}.
\begin{align}
\begin{split}
\label{set:uncertainty set}
\mathcal{F}_z(\hat{z}, \hat{\Sigma}, \hat\omega, \hat\gamma) = \{ z=\hat{z} + {\delta}: \\
\mathbb{E}[\delta]^T\hat{\Sigma}^{-1}\mathbb{E}[\delta] \leqslant \hat\omega, \mathbb{E}(\delta \delta^T) \preceq \hat\gamma \hat{\Sigma}\},\end{split}
\end{align}
where $\hat{\Sigma}$ is the estimate covariance of $\delta$, $\hat\omega$ and $\hat{\gamma}$ are two estimated constraints parameters.\end{defn} Definition~\eqref{set:uncertainty set} means that the distributional uncertainty set $\mathcal{F}_z$ relies on the estimated covariance matrix of prediction error $\delta$. The mean of $\delta$ is supposed to lie in an ellipsoid. The estimate covariance $\hat{\Sigma}$ lies in a positive semi-definite cone defined with a matrix inequality.

\begin{algorithm}
\caption{Uncertainty Set Parameters Estimation}
\label{alg: estimation}
 \KwData{a set of samples $Z$, bootstrap time $A$, bootstrap sample set size $M$, inner-bootstrap time $B$ and inner-bootstrap sample set size $N$, a significance level $\alpha$, a confidence parameter $\eta$, a prediction model $f$ and a observation set $O$.}
 \KwResult{Uncertainty Set Parameters and Confidence Regions}
 \While{i in 1 to A}{
    \While{j in 1 to B}{
        Re-sample $Z_j^i = \{ \tilde{z}_{j1}^i, ..., \tilde{z}_{jN}^i \}$ from $Z$ with replacement, estimate parameters of a prediction model $f_j^i(O)$, calculate the estimation residual set $\tilde{\Delta}_j^i = \{ \tilde{\delta}_{j1}^i, ..., \tilde{\delta}_{jN}^i \}$, where $\tilde{\delta}_{jk}^i = \tilde{z}_{jk}^i - \hat{z}_{jk}^i $. Compute its sample mean $\bar{\delta}_j^i$ and sample covariance $\bar{\Sigma}_j^i$.
    }
    Compute estimated covariance $\hat{\Sigma}^i = \frac{1}{B} \sum_{j = 1}^{B} \bar{\Sigma}_j^i$. \\
    \While{j in 1 to B}{
        Compute ${\omega}_{j}^i$ and ${\gamma}_{j}^i$ according to \eqref{def_omega} and \eqref{def_gamma} get ${\Omega}^i = \{ {\omega}_{1}^i, ..., {\omega}_{B}^i\}$ and ${\Gamma}^i = \{ {\gamma}_{1}^i, ..., {\gamma}_{B}^i\}$
    }
    Compute ${\omega}_{\alpha}^{i}$ from ${\Gamma}^{i}$ and ${\gamma}_{\alpha}^{i}$ from ${\Gamma}^{i}$\\
    \While{k in 1 to C}{
        Sampling with replacement from set ${\Omega}^i$ and ${\Gamma}^i$ to get bootstrap sample sets ${\Omega}^{ik}$ and ${\Gamma}^{ik}$\;
        Compute ${\omega}_{\alpha}^{ik}$ from ${\Omega}^{ik}$ and ${\gamma}_{\alpha}^{ik}$ from ${\Gamma}^{ik}$
    }
 }
Compute sample standard deviations: ${s}^i_\omega$, ${s}^i_\gamma$, $s_\omega$, $s_\gamma$, estimated parameters: $\hat{\omega}_{\alpha}$, $\hat{\gamma}_{\alpha}$, $\hat{\Sigma}$, quantiles: $q_{\eta/2}^\omega$ and $q_{1-\eta/2}^\omega$, $q_{\eta/2}^\gamma$ and $q_{1-\eta/2}^\gamma$
\end{algorithm}
\vspace*{-10pt}

\subsection{Uncertainty Set Construction Algorithm}
We develop Algorithm~\ref{alg: estimation} to solve Problem 1. Algorithm~~\ref{alg: estimation} computes the constraint parameters based on the bootstrap sample method~\cite{dixon2006bootstrap}. Given a pre-determined prediction model, historical data and observation data, in the outer loop, we sample with replacement from the historical sample set to train prediction model, then compute the set of prediction residual $\tilde\Delta$, i.e., realizations of prediction error, compute the sample mean and sample covariance. Then we compute constraint parameters by solving \eqref{def_omega} and \eqref{def_gamma}.
\begin{align}
    &{\omega}^i_{j} = [\bar{\delta}^i_j]^T(\hat{\Sigma}^{i})^{-1}[\bar{\delta}^i_j]
    \label{def_omega}\\
    &\underset{{\gamma}^i_{j}}{\text{min.}} {\gamma}^i_{j}\quad\text{s.t}\quad  \bar{\Sigma}^i_j \preceq {\gamma}^i_{j} \hat{\Sigma}^i,
    \label{def_gamma}
\end{align}
where $\bar{\delta}^i_j$ is a sample mean; $\bar{\Sigma}^i_j$ is a sample covariance; $\hat{\Sigma}^i = \sum_{j=1}^{N} \bar{\Sigma}_j^i$ is the estimated covariance in the $i$-th outer loop; $i \in \{1, 2, ..., A\},j \in \{1, 2, ..., B\}$ denote the outer loop and inner loop numbers, respectively. Now we compute the $\alpha$ percentiles for the sets $\Omega^i = \{ {\omega}^i_j\}_{j \in \{1, 2, ..., B\} }$, ${\Gamma}^i = \{ {\gamma}^i_j \}_{j \in \{1, 2, ..., B\} }$, which we denote as $\omega^i_{\alpha}, \gamma^i_{\alpha}$ respectively. We then sample with replacement from ${\Omega}^i$ and ${\Gamma}^i$ respectively to get constraint parameter sample sets ${\Omega}^{ik}$ and ${\Gamma}^{ik}$ for $C$ times, where $k \in \{ 1, ..., C\}$ denotes loop numbers. We also compute the $\alpha$ percentiles ${\omega}^{ik}_{\alpha}$ for $\Omega^{ik}$, ${\gamma}^{ik}_{\alpha}$ for $\Gamma^{ik}$.

Finally we get estimated constraint parameters $\hat{\omega}_{\alpha}$, $\hat{\gamma}_{\alpha}$ and estimated covariance $\hat{\Sigma}$ by \eqref{def_solve_problem}.
\begin{align}
    \hat{\omega}_{\alpha} = \sum_{i = 1}^{A} {\omega}_{\alpha}^i/A, 
    \hat{\gamma}_{\alpha} = \sum_{i = 1}^{A} {\gamma}_{\alpha}^i/A, 
    \hat{\Sigma} =  \sum_{i = 1}^{A} \hat{\Sigma}^i/A.
    \label{def_solve_problem}
\end{align}

Then we compute the $\eta/2$ and $1-\eta/2$ quantiles: $q_{\eta/2}^\omega$ and $q_{1-\eta/2}^\omega$ for $\{ \frac{{\omega}_{\alpha}^i - \hat{\omega}_{\alpha}}{{s}^i_\omega} \}_{i \in \{ 1,...,A\} }$, $q_{\eta/2}^\gamma$ and $q_{1-\eta/2}^\gamma$ for $\{ \frac{{\gamma}_{\alpha}^i - \hat{\gamma}_{\alpha}}{{s}^i_\gamma} \}_{i \in \{ 1,...,A\}}$, where ${s}^i_\omega$ and ${s}^i_\gamma$ are standard deviations on the set of $\{{\omega}_{\alpha}^{ik}\}_{k \in \{ 1, ..., C\}}$ and $\{{\gamma}_{\alpha}^{ik}\}_{k \in \{ 1, ..., C\}}$, respectively. We define the confidence regions of ${\omega}_{\alpha}$ and ${\gamma}_{\alpha}$: $[ \omega_l, \omega_u]$ and $[\gamma_l, \gamma_u]$, with the lower and upper bounds in \eqref{def_con_region},
\begin{align}
    \begin{split}
        \label{def_con_region}
        \omega_l = \hat{\omega}_{\alpha} - s_\omega q_{1-\eta/2}^\omega ,\quad
        \omega_u = \hat{\omega}_{\alpha} - s_\omega q_{\eta/2}^\omega ,\\
        \gamma_l = \hat{\gamma}_{\alpha} - s_\gamma q_{1-\eta/2}^\gamma ,\quad
        \gamma_u = \hat{\gamma}_{\alpha} - s_\gamma q_{\eta/2}^\gamma ,
    \end{split}
\end{align}
where $s_\omega$ and $s_\gamma$ are respectively standard deviations on the set of $\{{\omega}_{\alpha}^{i}\}_{i \in \{ 1, ..., A\}}$ and $\{{\gamma}_{\alpha}^{i}\}_{i \in \{ 1, ..., A\}}$. These two confidence regions are guaranteed to contain the true constraint parameters ${\omega}_{\alpha}$ and ${\gamma}_{\alpha}$ respectively, at the $1-\eta$ confidence level. 

Lemma~\ref{lem_alg} provides the theoretical support to the computing procedure of confidence regions. Thus, the produced uncertainty set parameters are guaranteed to be contained in some confidence regions with a given probability.

\begin{lem}
\label{lem_alg}
Given a pre-selected probability $1-\eta$, parameters calculated by 
Algorithm~\ref{alg: estimation} satisfy that ${\omega}_{\alpha}$ is contained between $\hat{\omega}_{\alpha} - s_\omega q_{\eta/2}^\omega$ and $\hat{\omega}_{\alpha} - s_\omega q_{1-\eta/2}^\omega$, ${\gamma}_{\alpha}$ is contained between $\hat{\gamma}_{\alpha} - s_\gamma q_{\eta/2}^\gamma$ and $\hat{\gamma}_{\alpha} - s_\gamma q_{1-\eta/2}^\gamma$.
\end{lem}
\begin{proof}
See Appendix~\ref{proof_lem_alg}
\end{proof}

%% file: sec_05theory.tex
\section{Computationally Tractable Form}
\label{sec:theory}
In this section, we derive the main theoretical result of this work: Theorem~\ref{thm:theory1}, a computationally tractable and equivalent convex optimization form for the distributionally robust optimization problem~\eqref{opt_final} via strong duality. The objective function~\eqref{def_final_obj} is convex over the decision variables and linear (concave) over the random parameter, with decision variables on the denominators. Though constraints are affine or convex over decision variables, most of them contain random parameters. This form is not a Linear Programming (LP)~\cite{vanderbei2015linear} or Semidefinite Programming (SDP)~\cite{yurtsever2021scalable}. The main process of deriving an equivalent convex optimization problem is to analyze~\eqref{def_final_obj} and constraints part. 
Based on Theorem~\ref{thm:theory1}, the optimal solution to~\eqref{opt_final} can be calculated in real time considering both passenger demand and EV supply uncertainties.

\begin{thm}
\label{thm:theory1}
	The distributionally robust optimization problem~\eqref{opt_final} with two distributional sets~\eqref{set:uncertainty set} is equivalent to the following convex optimization problem
	\begin{align}
	\label{obj_theory}
	&\min_{\substack{X^{1:\tau}, Y^{1:\tau}, D^{1:\tau}, U^{1:\tau};\\ S^{1:\tau}, V^{2:\tau}, O^{2:\tau},L^{2:\tau}; \\ {\lambda, Q_r,q_r,v_r,t_r,Q_c,q_c,v_c,t_c}}}
			\quad  H_o + v_r + t_r + v_c + t_c \nonumber\\
		\begin{split}
			&\text{s.t.}
			\quad  \begin{bmatrix}v_r & \frac{1}{2}(q_r+\lambda_U + \lambda_D)^T\\ \frac{1}{2}(q_r+\lambda_U + \lambda_D) & Q_r 
			\end{bmatrix} \succeq 0,\\
			\bigskip
			\bigskip
			&\quad \quad \begin{bmatrix}v_c & \frac{1}{2}(q_c+\lambda_V - Z)^T\\ \frac{1}{2}(q_c+\lambda_V - Z) & Q_c 
			\end{bmatrix} \succeq 0,\\
			\bigskip
			&\quad \quad t_r \geqslant (\hat\gamma_{r} \hat{\Sigma}_r+\hat{r}\hat{r}^T)\cdot Q_r + \hat{r}^T q_r\\ &\quad \quad \quad +\sqrt{\hat\omega_{r}} \|\hat{\Sigma}_r^{1/2} (q_r+2Q_r\hat{r})\|_2 ,\\
			\bigskip
			&\quad \quad t_c \geqslant (\hat\gamma_{c} \hat{\Sigma}_c+\hat{c}\hat{c}^T)\cdot Q_c + \hat{c}^T q_c\\ &\quad \quad \quad +\sqrt{\hat\omega_{c}} \|\hat{\Sigma}_c^{1/2} (q_c+2Q_c\hat{c})\|_2 ,\\
			\bigskip
			&\quad \quad Q_r, Q_c \succeq 0 ,
			\quad \quad \lambda, v_r, v_c \geq 0 ,\\
			\bigskip
			&\quad \quad z_i^k \geqslant \frac{1}{(T_i^k)^a}\\
			\bigskip
			&\quad \quad x_{ij}^k \geqslant 0 \text{ and } x_{ij}^k = 0 \text{ when } w_{ij} \geq m_1;\\
			\bigskip
            &\quad \quad y_{ij}^k \geqslant 0 \text{ and } y_{ij}^k = 0 \text{ when } w^*_{ij} \geq m_2,
		\end{split}
	\end{align}
	where $H_o=J_D -  ( \lambda_S^T f_S + \lambda_O^T f_O + \lambda_L^T f_L + \lambda_s^T S^{1:\tau} + \lambda_l^T L^{2:\tau}) - \lambda_D^T diag(lS^T - DD^T) - \lambda_U^T diag(hS^T + UU^T) - \sum\limits_{k = 1}^{\tau -1}(-V^{k+1}_i + \sum\limits_{j=1}^{N} P^k_{vji}S^k_j+ \sum\limits_{j=1}^{N} Q^k_{vji}O^k_j)\lambda_{V_i^k} $,
	$J_D$ is defined as~\eqref{obj_dist}.
\end{thm}
\begin{proof}
See Appendix~\ref{proof_theorem1}.
\end{proof}
This convex optimization problem~\eqref{obj_theory} has a objective which is a linear function of decision variables $v_r$, $t_r$, $v_c$, $t_c$, $X^{1:\tau}$, $Y^{1:\tau}$, $S^{1:\tau}$, $V^{2:\tau}$, $O^{2:\tau}$, $L^{2:\tau}$ and $\lambda$; as well as a quadratic function of decision variables $D^{1:\tau}$ and $U^{1:\tau}$. It includes linear matrix inequalities constraints, linear norm inequalities constraints and other convex constraints.


\begin{remark}[Computational Complexity]
In the DRO EV balancing problem, every uncertain parameter is in a convex and compact support set. Each support set is equipped with an oracle which has the following properties. First, given any value of the uncertain parameter, within polynomial time (with respect to the dimension of the uncertain parameter), the oracle can (i) either confirm the value is within the support set; (ii) or provide a hyperplane that separates the value from the support set \cite{grotschel1981ellipsoid}.
Second, given all parameters and variables, the value of $J_D$ and $J_E$ can be evaluated in polynomial time. Notice that the problem \eqref{obj_theory} has a linear objective function and convex constraint functions over decision variables. Hence, Proposition 1 in~\cite{delage2010distributionally} can be applied and the problem \eqref{obj_theory} can be solved to any precision $\epsilon$ within polynomial time of $\log(1/\epsilon)$ and the size of the problem. This equivalent computationally tractable convex optimization
problem can be solved by existing convex optimization
problem solvers such as CVXPY \cite{agrawal2018rewriting}, CVXOPT
\cite{andersen2013cvxopt}. For instance, CVXOPT converts a problem into its equivalent standard form known as conic form, and provide polynomial
interior-point algorithms \cite{nesterov1994interior} to solve it.
\end{remark}

\vspace*{-10pt}
\color{black}

%% file: sec_06experiments.tex
\vspace{-5pt}
\section{Experiment}
\vspace{-5pt}
\label{sec:experiment}


In this section, we evaluate the performance of the designed distributionally robust optimization-based EV balancing method with electric taxi (e-taxi) data from the Chinese city Shenzhen (one of the largest cities in China, which operates over 10,000 E-taxis). Six-week real-world data is utilized in the experiment, which includes 60GB EV GPS data, 5.5 GB transaction data from over 10,000 EVs. We split the whole dataset 67\%-33\% between the training set and the testing set. This 2:1 ratio is commonly used in train-test split for evaluating machine learning algorithms. Even though two weeks seem short for testing compared to the data used in other machine learning methods, the data size is large (20GB GPS and 1.9 GB transaction data involving over 10,000 EVs). \textcolor{black}{Due to security and privacy regulations and compliance, these datasets are the best we can get.} To our knowledge, few existing EV work utilized such large-scale data from large-scale EV deployments for algorithm evaluation.  \textcolor{black}{In Fig. \ref{fig_summary}, there is a high-level illustrative flow graphic to describe the proposed solution framework to help readers understanding the experimental procedures.}
\begin{figure}
	\centering
	\includegraphics [width=0.4\textwidth]{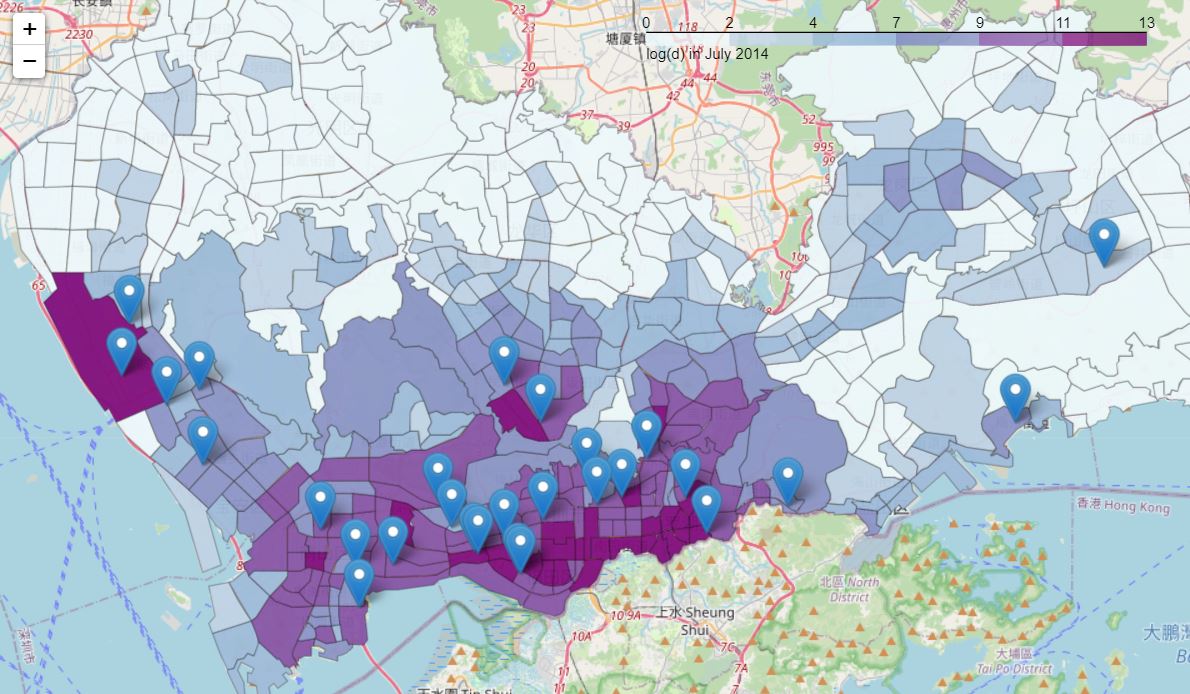}
	\vspace{-10pt}
	\caption{Heat map of demand in Shenzhen City: lighter means less demand. }
	\label{fig_heat_demand}
	\vspace{-10pt}
\end{figure}

\vspace*{-10pt}
\subsection{Data Description}
\vspace*{0pt}
\input{sup_03tab_data_demo}
In total, there are four different datasets used in this paper, including 
E-taxi GPS data (vehicle ID, locations, time and speed, etc), 
Transaction data (vehicle ID, pick-up and drop-off time, pick-up and drop-off location, travel distance, etc), 
charging station data (locations, name, the number of charging ports, etc), and 
city partition data (geographic boundaries of 491 small separate regions composing Shenzhen). Our datasets are obtained by collaborating with the Shenzhen Transportation Committee for its smart city initiative. An example of these four datasets is shown in the Tab.~\ref{tab_data}.

\textbf{GPS Data} include fields of the status of taxis, e.g., the vehicle ID, the GPS location longitude and latitude information, time-stamp, direction, current speed, etc.
\textbf{Transaction Data} describe each transaction information, e.g., vehicle ID, pick-up/drop-off time, pick-up/drop-off location (longitude \& latitude), and travel distance, etc.
\textbf{Charging Station Data} include the locations (i.e., GPS) of each charging station, station name, and the number of charging ports in each station, etc.
\textbf{Urban Partition Data} describe the urban partition for population census of the Shenzhen city. There are 491 regions in total, and each region has a region ID and longitudes \& latitudes of its boundary. 

Fig.~\ref{fig_heat_demand} is a heat map of demand in Shenzhen city. It shows the unbalanced distribution of mobility demand. There is higher passenger demand in downtown and airport areas than in suburb areas. This realistic situation verifies the necessity to provide fair service in EV AMoD systems. These blue markers in the Fig.~\ref{fig_heat_demand} denote charging station locations. The unbalanced charging stations distribution also enhances our motivation to consider the charging fairness in EV AMoD balancing problems.

\vspace{-10pt}
\subsection{Data Processing}

For the host machine and software used to clean and manipulate the original datasets, we utilize a 34 TB Hadoop Distributed File System (HDFS) on a cluster consisting of 11 nodes, each of which is equipped with 32 cores and 32 GB RAM. To make this data fit within the context of our problem formulation, we further clean our datasets according to several principles including continuity of time, homogeneity in data distributions, uniformly time granularity, etc. EV mobility and charging patterns is also considered in the filtering process and a comprehensive investigation of the process has been introduced in the literature~\cite{wang2020understanding}.

To merge multi-source datasets, we first select the electric taxis with complete GPS and the transaction records in the dataset. Then we insert the departure and arrival of each transaction as new GPS timestamps and have area information labelled, building an aligned timeline for each taxi. We further utilize a widely adopted spatiotemporal constraint-based method in the literature \cite{wang2019sharedcharging, tian2016real} to infer the charging activities taken by taxis based on their individual timeline. We also define 7 mutually exclusive labeling codes for marking timestamps' riding or charging status, based on which we can obtain the statistics of EV activities in both spatial and temporal dimensions.


\color{black}

\begin{figure}[h]
	\centering
	\vspace{-15pt}
	\includegraphics [width=0.5\textwidth]{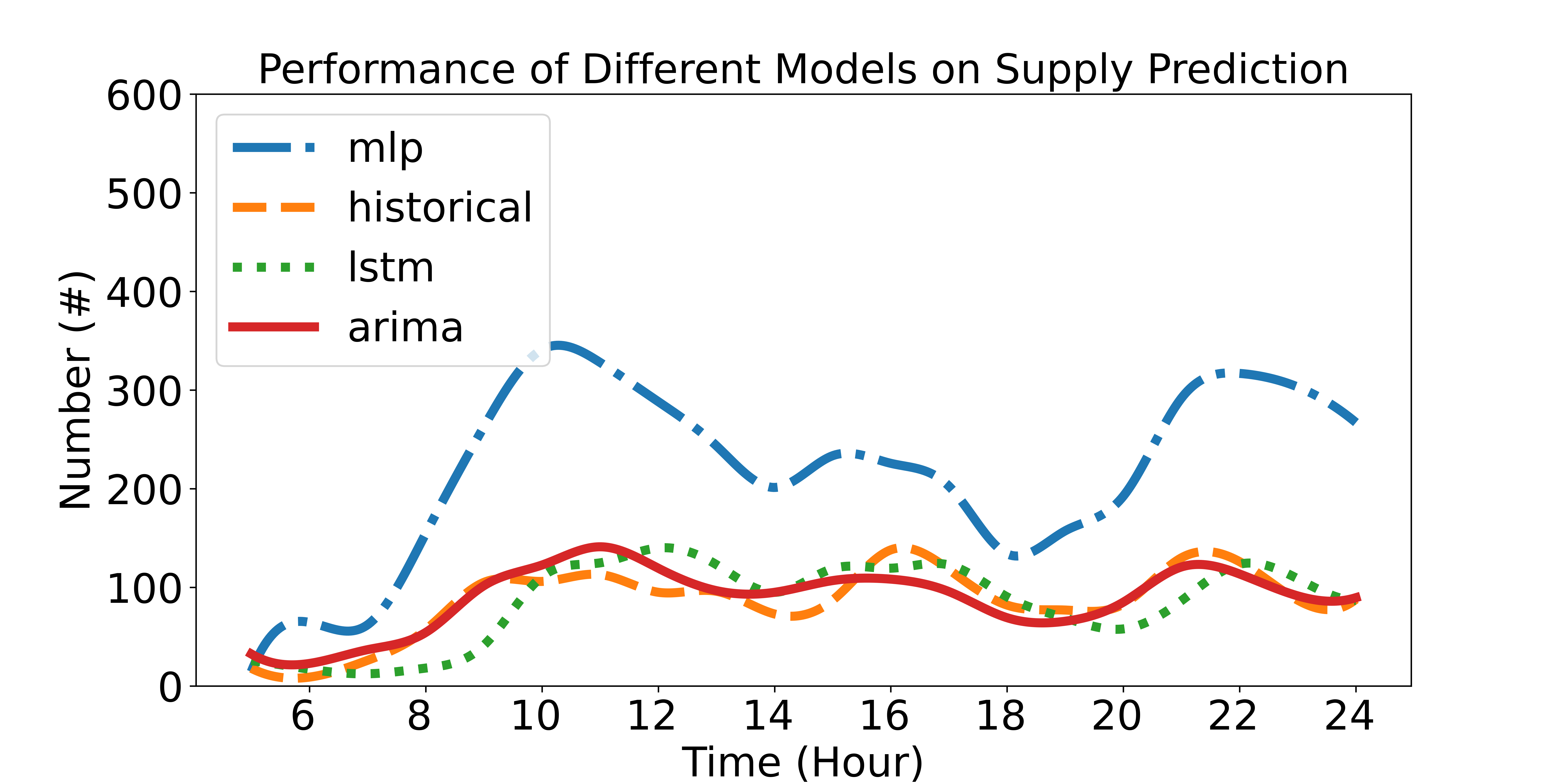}
	\vspace{-22pt}
	\caption{The performance of ARIMA model in predicting supply and demand is better than MLP and LSTM models. So we decide ARIMA model as the prediction model $f$ when using Algorithm~\ref{alg: estimation} to construct uncertainty sets.}
	\label{fig:arima_nn}
	\vspace{-15pt}
\end{figure}

\begin{figure}[!b] \centering
\vspace*{-15pt}
\includegraphics[width=0.5\textwidth, keepaspectratio=true]{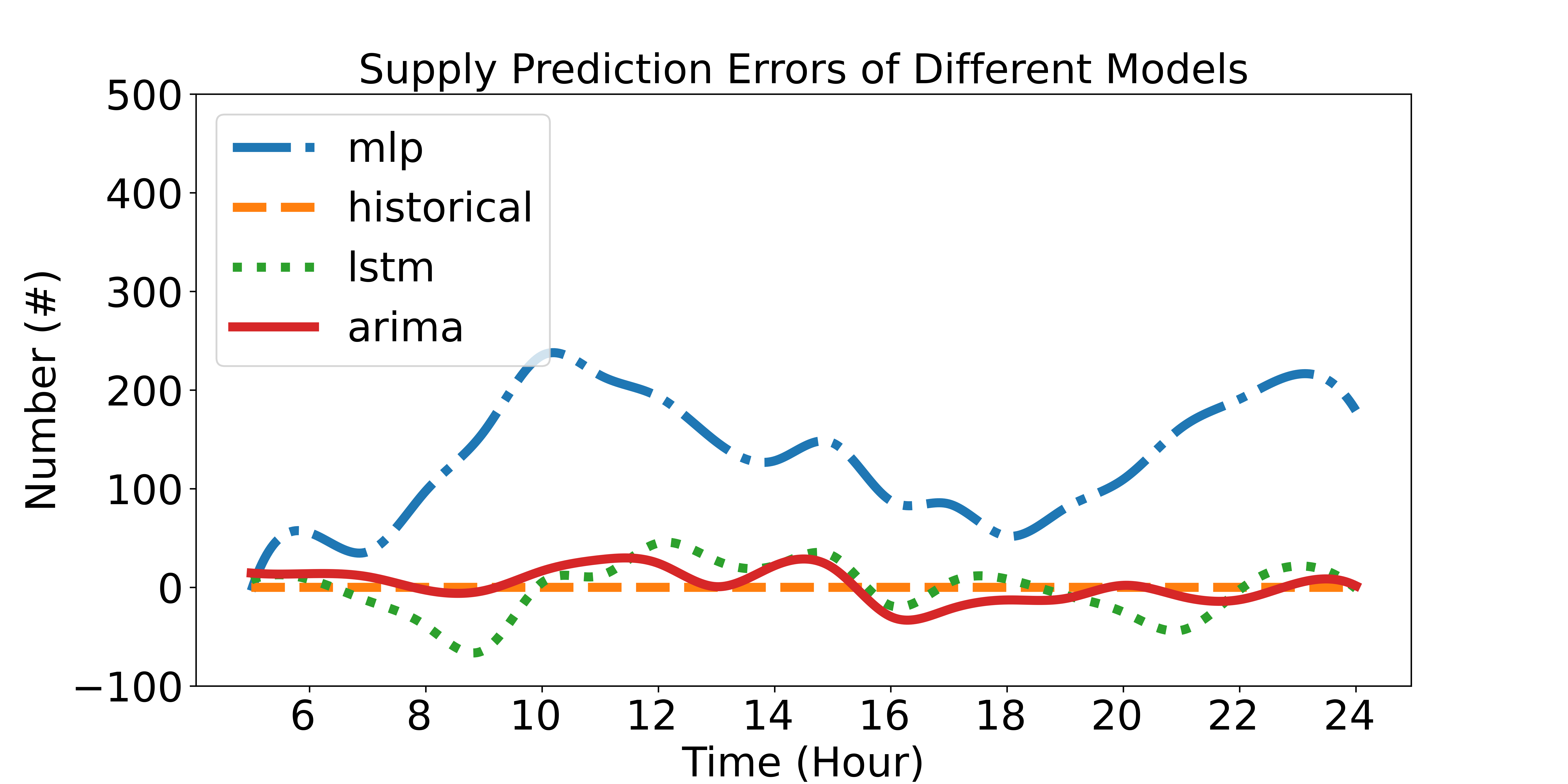}
\vspace*{-22pt}
\caption{The ARIMA model has the lowest supply prediction errors.}
\vspace*{-10pt}
\label{fig_error}
\end{figure}
\subsection{Prediction Model}
We decide which prediction model to be used by comparing the performance of these three models: Long Short-Term Memory model (LSTM)~\cite{du2019deep}, AutoRegressive Integrated Moving Average model (ARIMA)~\cite{giraka2020short} and Multi-Layer Perceptron model (MLP)~\cite{abbasi2021deep}. In Fig.~\ref{fig:arima_nn}, the three models all show the general trend of supply in one day. But ARIMA model performs better than other two models in term of prediction accuracy. In Fig.~\ref{fig_error}, the prediction errors for supply at each hour are reported. MLP model has the largest prediction errors since the blue curve is almost always higher than all the other curves. The curve of LSTM model is close to the ground truth (the zero horizon line) but still has larger errors than the ARIMA model. To quantitatively compare the performance of these models, we further provide the mean square error (MSE) in Tab. \ref{tab_mse}. MSE is the average squared difference between the predicted values and the historical values. It measures the prediction quality of a model and a lower MSE indicates a better accuracy. We can see that ARIMA model achieves the smallest MSE. Hence, we decide ARIMA model as the prediction model $f$ when using Algorithm~\ref{alg: estimation} to construct uncertainty sets. In Fig.~\ref{fig:arima}, we compare ARIMA model's predicted values and historical values of supply and demand on one day in region $2$. ARIMA model demonstrates supply and demand's time trends very well and the predicted values of demand in peak hours: 8am-10am, 2pm-4pm, 8pm-10pm, are close to historical values.

\vspace*{-5pt}
\subsection{Distributional Uncertainty Set}
Tab.~\ref{tab_compare_val} shows how the constraint parameters $\hat\omega_{c}$, $\hat\gamma_{c}$ change for different parameter $B$ of Alg.~\ref{alg: estimation}. The values of $\hat\omega_{c}$ and $\hat\gamma_{c}$ decrease when $B$ increases and the speed of decreasing turns slower as $B$ becoming larger. We also notice that as $B$ increases, the corresponding estimated constraint parameters $\hat\omega_{c}$ and $\hat\gamma_{c}$ have a trend to converge to a certain constant that meets bootstrapping algorithm's intuition. Here one time interval is $1$ hour, the time horizon $\tau$ is $2$, the significant $\alpha = 0.25$ and the true probability distributions of demand and supply variables $r$, $c$ are separately contained in the constructed distributional uncertainty sets with probability $75\%$.

\input{sup_06tabs_alg_performance}

\input{sup_04fig_exp_performance}

We validate the robustness of our proposed uncertainty set construction Alg.~\ref{alg: estimation} in Fig.~\ref{fig_variance}-1 and Fig.~\ref{fig_variance}-2. We run Alg.~\ref{alg: estimation} with different bootstrap time $A$ and the baseline algorithm~\cite{miao2021data} for 20 times, using the same prediction model and parameters. We use $A=0$ to denote the baseline algorithm. Then we compute the variances of the outputted uncertainty set parameters. Variance is an indicator of the robustness of random algorithms\cite{dietterich1995machine}. It is known that in several general cases, a random algorithm which produces a set of outputs with a smaller variance is more robust. Compared to the baseline algorithm, our Alg.~\ref{alg: estimation} produces outputted constraint parameters with smaller variances. In particular, our construction algorithm improve the robustness to randomness by at least 65\%, compared with the baseline algorithm. And the values of variance decrease as $A$ getting larger.

\vspace*{-5pt}
\subsection{Performance Comparison}
We compare our distributionally robust optimization (DRO) method with the baseline robust method~\cite{hao2020robust} and the non-robust method~\cite{yuan2019p2charing} by using the same real-time sensing data. The non-robust method is described in Appendix, please refer to Equation \eqref{non_robust_prob}. The EV balancing decision is made every hour, and the running time of our algorithm is within 5 seconds ($3.76 \text{ s } \pm 1.18$), which can satisfy the real-time requirement of making balancing decisions.
\begin{table}[]
\centering
\caption{Mean square error of different models}
\vspace*{-5pt}
\begin{tabular}{c|c|c|c}
\hline
    & MLP         & ARIMA       & LSTM        \\ \hline
MSE & 20148.21 & 254.65 & 647.45 \\ \hline
\end{tabular}
\label{tab_mse}
\vspace*{-15pt}
\end{table}
Fig. \ref{fig:total_cost} shows the total balancing cost $M_b^k$ from 5am to 11pm using our method \eqref{opt_final} as well as two baseline methods. The metric $M_b^k$ is formally given in \eqref{metric_blancing}, which is a weighted sum of the vacant and low-battery EVs' moving distance after executing balancing decisions at time $k$. A smaller $M_b^k$ is better, since a smaller total balancing cost indicates higher system efficiency and lower distance that the EVs run without serving passengers.
\begin{align}
\label{metric_blancing}
    M_b^k &= J_D(X^{k}, Y^k) = \sum\limits_{i=1}^{N} \sum\limits_{j=1}^{N} (x^k_{ij} w_{ij} + \beta y^k_{ij} w^*_{ij}).
\end{align}
We can see that most of the time, the total cost of our method is lower than that of the two baseline methods. In particular, the average total balancing cost is reduced by 14.49\% compared with the non-robust method and reduced by 7.37\% compared with the baseline robust method. 
In Fig. \ref{fig_variance}-3, we also show the variance of the optimal daily balancing cost of the DRO~\eqref{opt_final} solutions when using Alg. 1 with different bootstrap parameters $A$ and the baseline uncertainty set construction algorithm \cite{miao2021data} ($A$ = 0). For an uncertainty set produced in different settings, we keep other parameters the same and run our DRO method for $20$ times. Then we compute the variance of daily balancing cost $\sum_{k=1}^K M_b^k$. The variance of the daily balancing cost decreases as $A$ gets larger. And when using our Alg.~\ref{alg: estimation}, we obtain a lower variance of daily balancing cost, compared with the baseline algorithm. In particular, the variance is reduced by at least 31\%, compared with the baseline algorithm. Our DRO method achieves more stable balancing cost, by using together with our proposed uncertainty set construction algorithm. 

In Fig.~\ref{fig:fairm} and Fig.~\ref{fig:fairc}, we show the charging fairness $M_c^k$ and mobility fairness $M_m^k$ of executing different balancing decisions from 5am to 11pm, respectively. As the fairness metrics defined in \eqref{metric_charging} and \eqref{metric_mobility}, the charging fairness $M_c^k$ is a negative sum of total absolute difference between the global and local charging supply-demand ratios at time $k$. The mobility fairness $M_m^k$ has a similar definition. For both of them, larger values means higher fairness, which are preferable. And higher fairness indicates better balanced resource allocation and similar service quality for customers among different regions and time frames in the city. 
\begin{align}
\label{metric_charging}
    M_c^k &= -J_1(Y^{k}) = -\sum_{i \in \sigma}\left| \frac{c^k_i}{T^k_i} - \frac{\sum_{j \in \sigma}c^k_j}{\sum_{j \in \sigma} T^k_i}\right|;\\
    \label{metric_mobility}
        M_m^k &= -J_M(X^{k}) = -\sum_{i=1}^{N}\left| \frac{r^k_i}{S^k_i} - \frac{\sum_{j=1}^{N}r^k_j}{\sum_{j=1}^{N} S^k_j}\right|.
\end{align}
By using our DRO method, the average mobility fairness and charging fairness is improved by 15.78\% and 34.51\%, respectively, compared to the non-robust method, and improved by 10.45\% and 30.92\%, respectively, compared to the baseline robust method. Our DRO method outperforms the non-robust method because it considers both supply and demand uncertainties when making EV balancing decisions. The DRO problem formulation, uncertainty set construction based on data, and equivalent convex optimization form derivation procedures are all designed carefully to solve the challenge in a computationally tractable way. In contrast, the baseline methods in the literature either does not consider model uncertainties such as the non-robust methods, or only considers one type of uncertainty such as the robust optimization method with the demand uncertainty.
\begin{figure}[!htb] \centering
\vspace*{-15pt}
\includegraphics[width=0.5\textwidth, keepaspectratio=true]{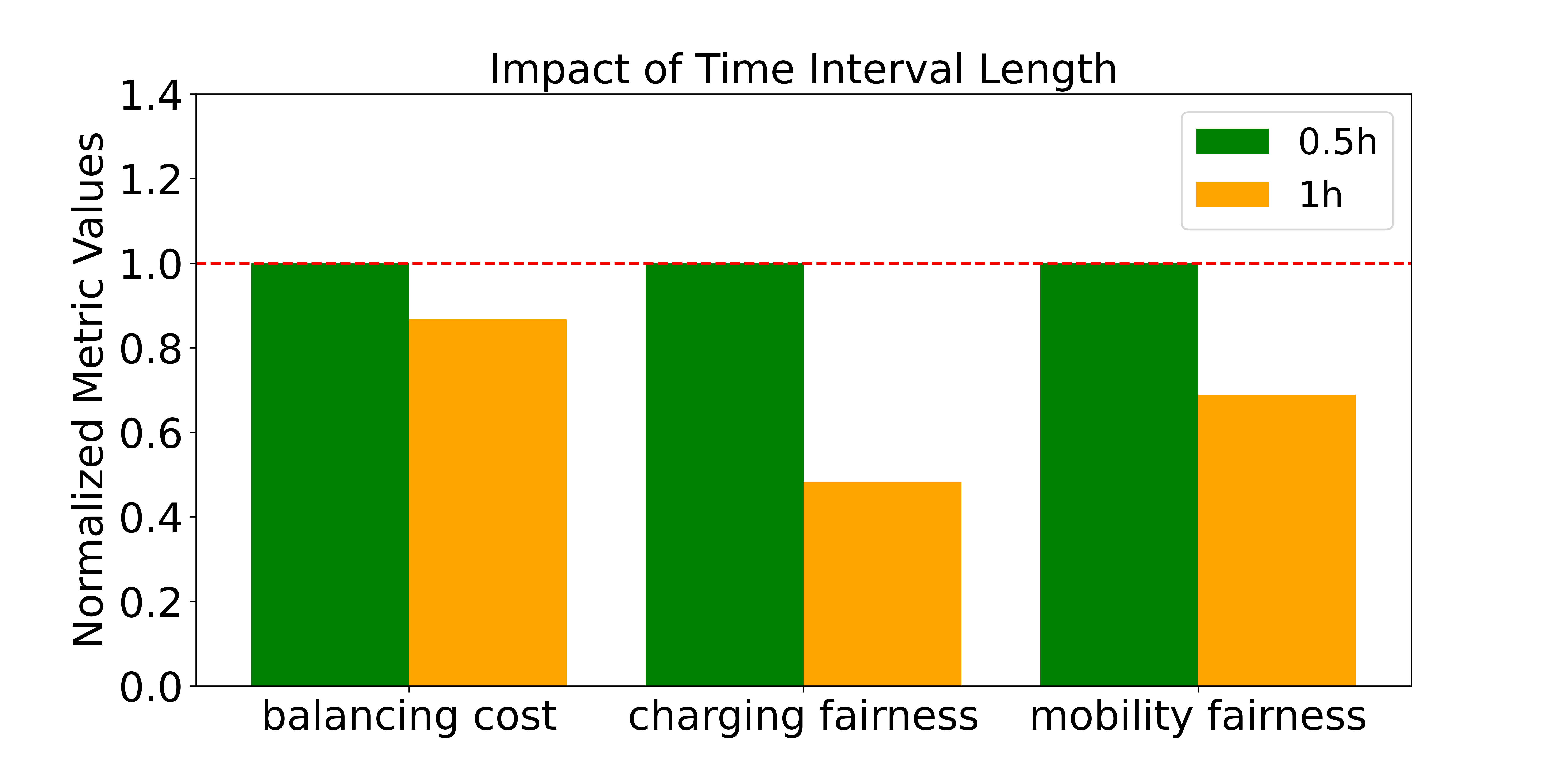}
\vspace*{-22pt}
\caption{With a shorter time interval length, the DRO solution results a higher total balancing cost in the same time frame, but better mobility and charging fairness.}
\vspace*{-10pt}
\label{fig_time_interval}
\end{figure}

In Fig. \ref{fig_time_interval}, we show the effect of different time interval lengths on our DRO method. We compare the total balancing cost, the mobility and charging fairness when using 0.5 hour and 1 hour as the time interval length, respectively. The metric values of the balancing cost, the mobility and charging fairness shown in Fig. \ref{fig_time_interval} are normalized within $[0,1]$ such that the highest result is fixed at value 1. With a shorter time interval l0,1ength, the DRO solution results a higher total balancing cost (a larger value in Fig.~\ref{fig_time_interval}) in the same time frame, but better mobility and charging fairness. One interpretation is that there exists trade off between different objectives for balancing EVs in the AMoD system. For instance, the algorithm can sacrifice balancing cost to get higher charging and mobility fairness by balancing the vehicles more frequently.

\begin{remark}[Practical implementation to address ITS open problems]
Our proposed DRO method can be applied to the control of EV AMoD systems such as EV balancing and charge scheduling problems. Vehicle balancing is an important objective of EV AMoD systems \cite{ammous2018optimal, wang2019sharedcharging, wang2020understanding}. By repositioning of customer-free EVs, it aims to minimizing the imbalance of the EV distribution caused by asymmetrical transportation demand \cite{zardini2022analysis}. Our method achieves lower balancing cost and higher mobility and charging fairness, thus contributes to congestion mitigation and transportation efficiency. Further, it is verified by using real world data thus able to provide insights and helps to autonomous taxis operation companies in their decisions.
\end{remark}
\vspace*{-5pt}
\begin{remark}[Limitations] (i) In this paper, we did not consider the potential impacts of political policies such as tax and subsidy in EV AMoD systems. For example, the EV promotion in Shenzhen is strongly supported by the government \cite{li2020shenzhen}. However, some literature has found that though tax and subsidy affect EV fleet evolution, there are little impacts on the EVs mobility and charging patterns \cite{wang2020understanding}. Therefore, political factors are not emphasized in this work. (ii) The potential impacts of private EVs in AMoD systems are not considered. For instance, private EVs may share charging stations with autonomous EVs. However, according to some field studies in Shenzhen \cite{wang2019sharedcharging}, few private EVs prefer utilizing fast charging stations because they have no need to leverage fast charging like commercial EVs for keeping normal business activities. Therefore, we did not explicitly include the impact of private EVs in AMoD balancing decisions. We do update the status and available spots of charging stations before making EV charging decisions in the designed DRO method, to mitigate the effects of other EVs that share the charging resource with the EV AMoD system. \textcolor{black}{(iii) Our model shows good performance on the six-week data, and we will try to test it on datasets of longer time duration after we can have access to them. }
\end{remark}
\color{black}

%% file: sup_03tab_data_demo.tex
\begin{table*}[!ht]\centering
\vspace*{0pt}
\caption{An Example of the four datasets}
\vspace*{-6pt}
\footnotesize
\begin{tabular}{ccccccc}
\hline
{GPS Data} & plate ID & longitude & latitude & time & speed (km/h) \\ \cline{2-6}
 & BDXXXX & 114.0121 & 22.526104 & 2015-08-16 08:30:43 & 35 \\ \hline \hline
{Transaction Data} & plate ID & pickup time & dropoff time & pickup location & travel distance (m) \\ \cline{2-6}
 & BDXXXX & 2015-09-03 13:47:58 & 2015-09-03 13:57:23 & (113.9867, 22.5433) & 6954 \\ \hline \hline
{Charging Station Data}& station ID & station name & longitude & latitude & number of charging ports \\ \cline{2-6}
 & 30 & NB0005 & 113.9878608 & 22.55955418 & 40 \\ \hline \hline
{Urban Partition Data}& Region ID & Longitude1 & Latitude1  & Longitude2 & Latitude2 \\ \cline{2-6}
 & 1 &114.31559657&22.78559093&114.311230763&22.78220351 \\ \hline
\end{tabular}
\label{tab_data}
\vspace{-10pt}
\end{table*}

%% file: sup_06tabs_alg_performance.tex
\begin{table}[ht]\centering
\vspace*{-10pt}
\caption{Thresholds $\hat{\omega}_{c}$ and $\hat{\gamma}_{c}$ for Different Parameters $B$}
\vspace*{-6pt}
\begin{tabular}{c|c|c|c|c|c|c}
\hline
$B$     & 10    & 20    & 50    & 100   & 500   & 1000  \\ \hline
$\hat\omega_{c}$ & $1.504$ & $0.964$ & $0.576$ & $0.399$ & $0.296$ & $0.176$ \\ \hline
$\hat\gamma_{c}$ & $3.715$ & $2.832$ & $2.006$ & $1.768$ & $1.374$ & $1.317$ \\ \hline
\end{tabular}
\label{tab_compare_val}
\vspace{-5pt}
\end{table}



%% file: sup_04fig_exp_performance.tex
\begin{figure}[!t]
	\centering
	\vspace*{-10pt}
	\includegraphics [width=0.5\textwidth]{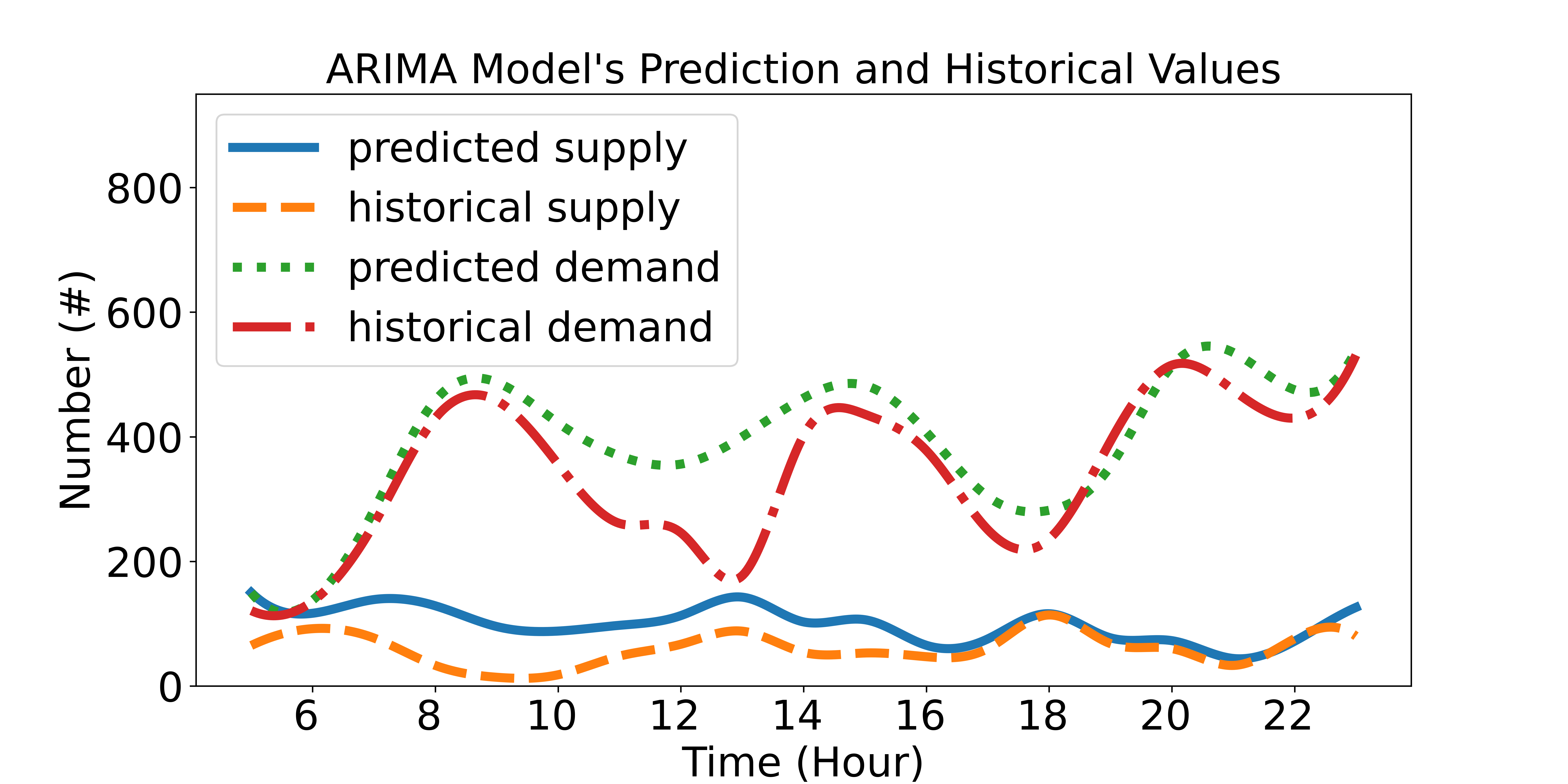}
	\vspace{-22pt}
	\caption{ARIMA model demonstrates time trends very well.}
	\label{fig:arima}
	\vspace{-15pt}
\end{figure}

\begin{figure}
	\centering
	\vspace*{-5pt}
	\includegraphics [width=0.5\textwidth]{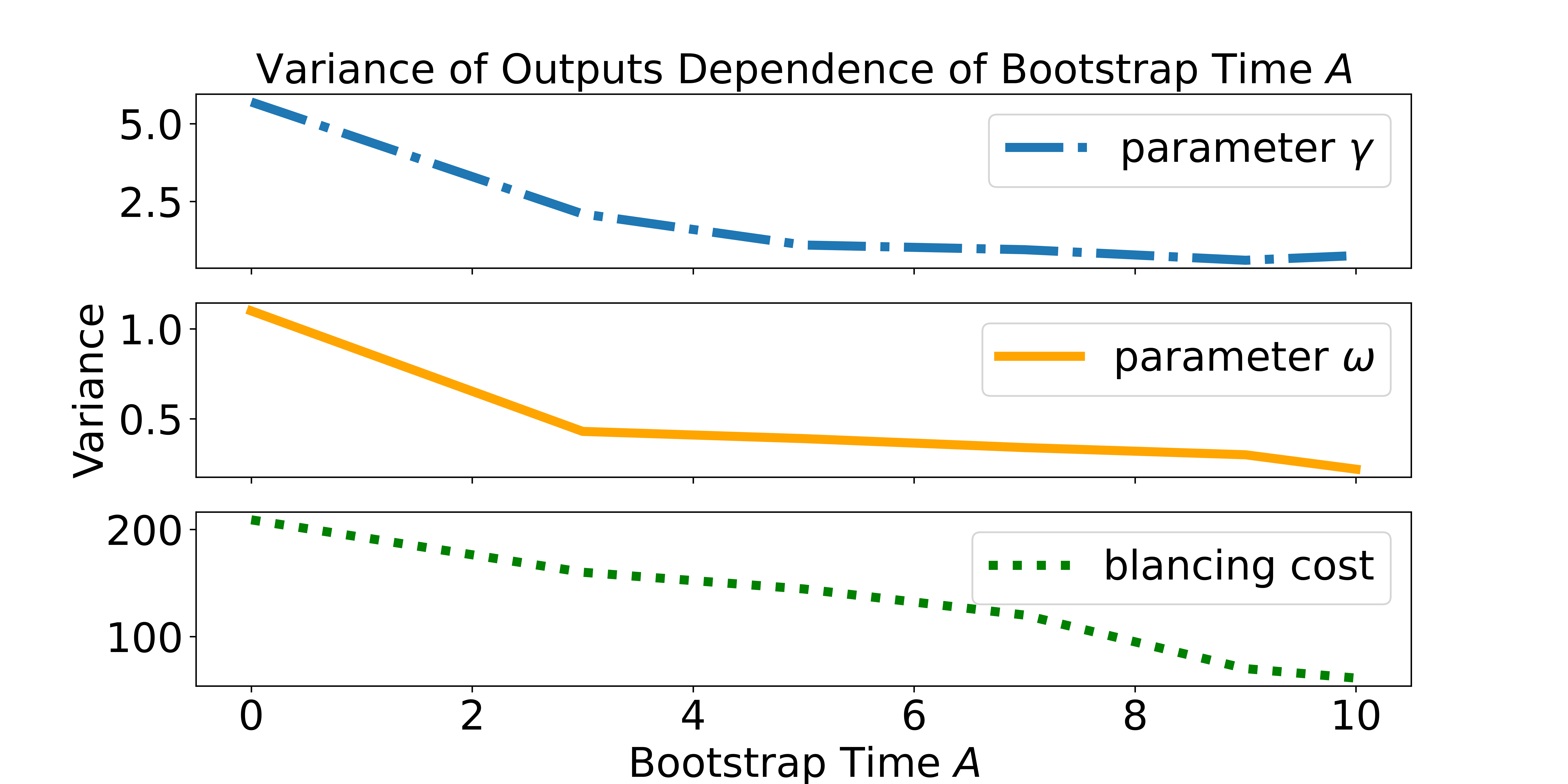}
	\vspace{-22pt}
	\caption{The variances of constraint parameters/balancing cost are decreased as the bootstrap time $A$ increasing.}
	\label{fig_variance}
	\vspace{-5pt}
\end{figure}

\begin{figure}[!ht]
	\centering
	\vspace*{-10pt}
	\includegraphics [width=0.5\textwidth]{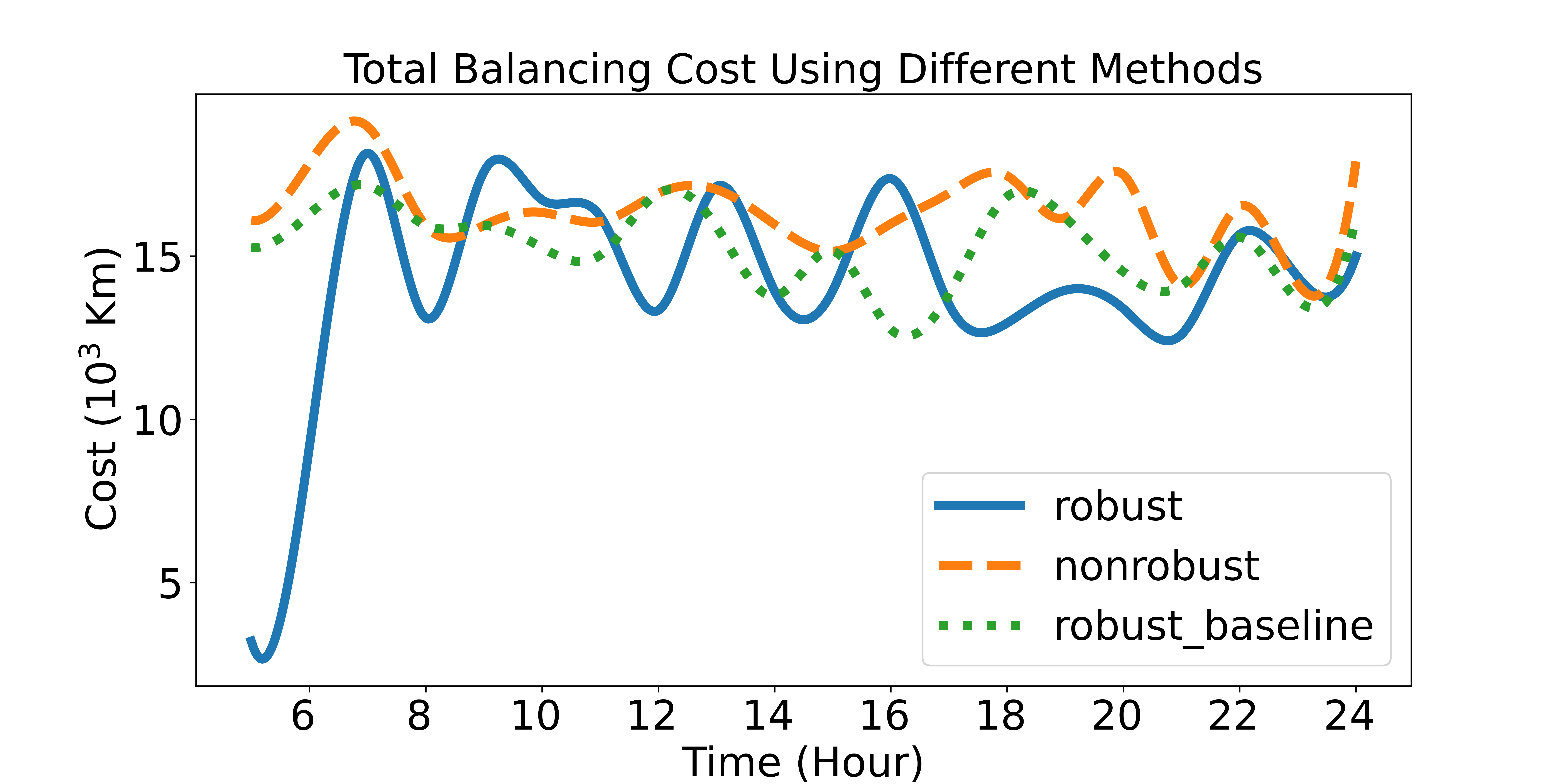}
	\vspace{-22pt}
	\caption{By using our DRO method, the average total balancing cost is reduced by 14.49\% compared to non-robust method.}
	\label{fig:total_cost}
	\vspace{-10pt}
\end{figure}

\begin{figure}
	\centering
	\vspace*{-10pt}
	\includegraphics [width=0.5\textwidth]{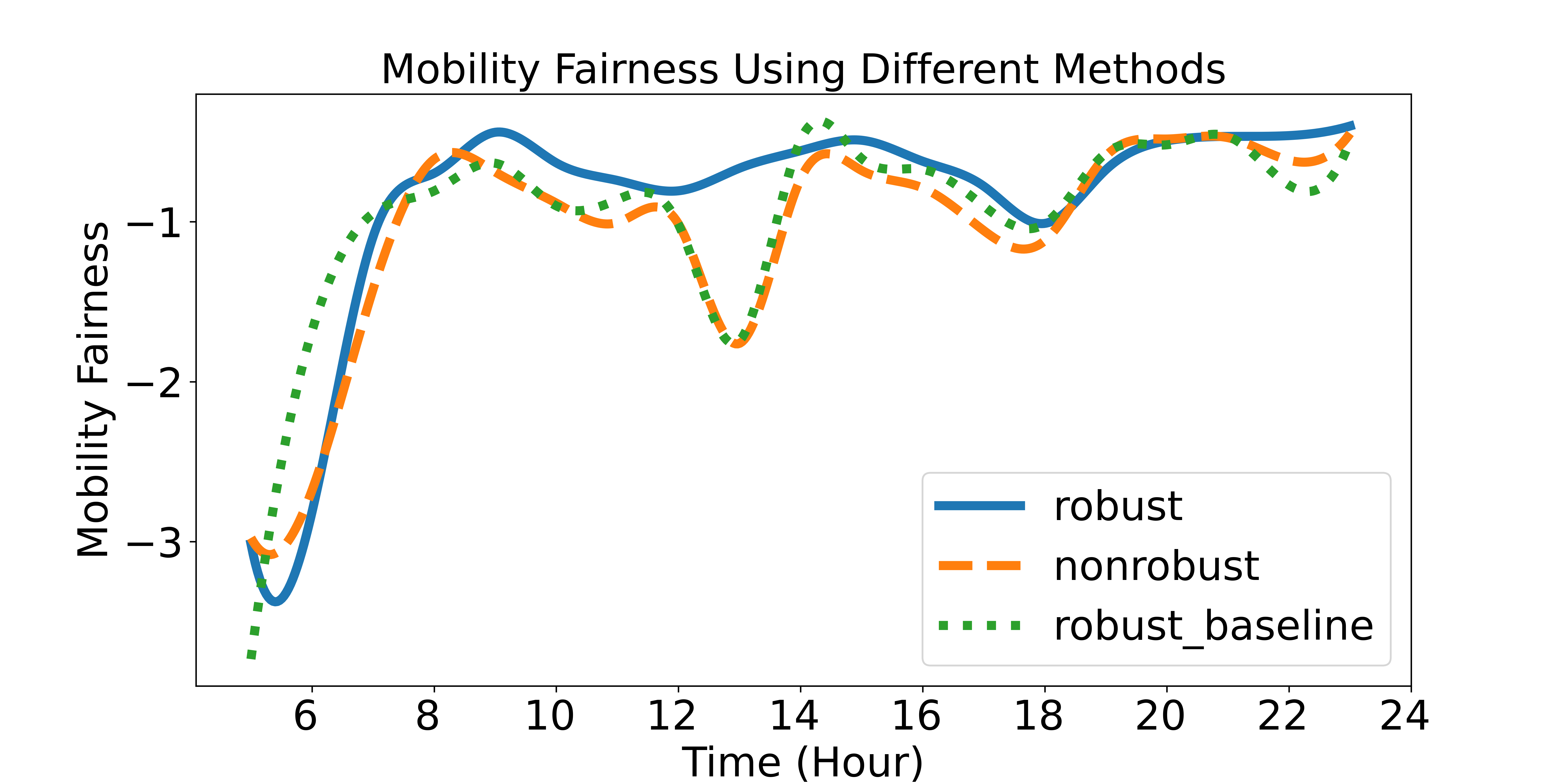}
	\vspace{-20pt}
	\caption{By using our DRO method, the average fairness of mobility supply-demand ratio is improved by 15.78\% compared to non-robsut method.}
	\label{fig:fairm}
	\vspace{-10pt}
\end{figure}

\begin{figure}[!hb]
	\centering
	\vspace{-5pt}
	\includegraphics [width=0.5\textwidth]{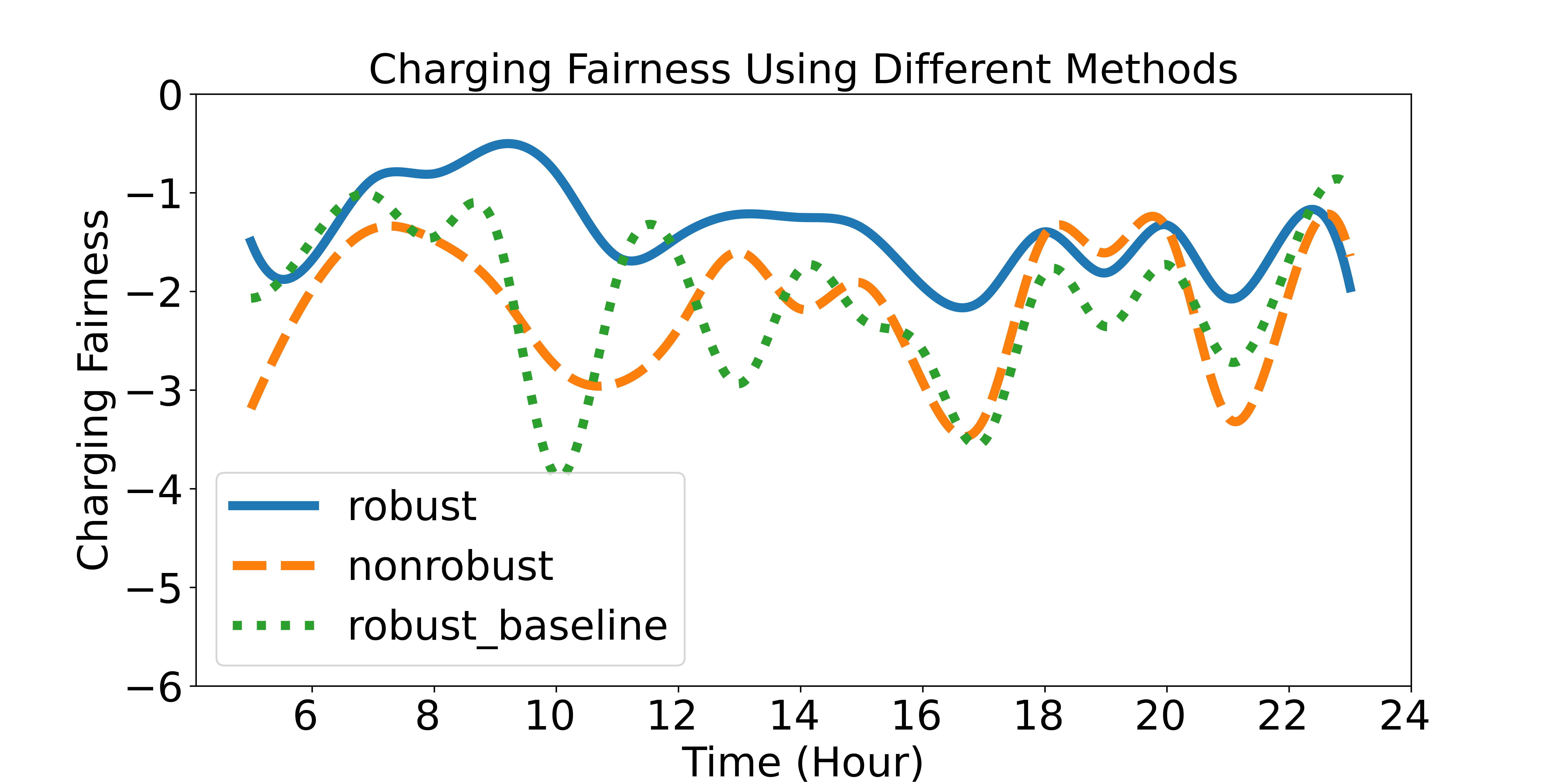}
	\vspace{-20pt}
	\caption{By using our DRO method, the average fairness of charging supply-demand ratio is improved by 34.51\% compared to non-robsut method.}
	\label{fig:fairc}
	\vspace{-10pt}
\end{figure}

%% file: sec_07conclusion.tex
\vspace{-10pt}
\section{Conclusion}
\vspace{-5pt}
\label{sec:conclusion}

Autonomous mobility-on-demand systems can provide efficient transportation services. However, with an increasing number of EVs, it is still challenging to improve the efficiency of AMoD systems under EVs supply uncertainty, due to limited charging facilities in cities and complicated charging dynamics. In this paper, we design a data-driven distributionally robust EV balancing method to minimize the worst-case expected cost under uncertainties of both passenger mobility demand and EV supply. In addition to reducing total vehicle balancing costs, we also balance the mobility and charging supply-demand ratios of different regions in the city. We propose efficient algorithms to construct distributionally uncertainty sets of the predicted mobility demand and EV supply. Then we derive an equivalent computationally tractable form of the distributionally robust EV balancing problem under the ellipsoid uncertainty sets constructed from historical data. Evaluations based on real-world E-taxi data show that the average total balancing cost is reduced by 14.49\%, and the average passenger mobility fairness and EV charging fairness are improved by 15.78\% and 34.51\%, respectively. In the future, we will further evaluate our algorithm based on large-scale data from several years from multiple cities. 

%% file: sec_08appendix.tex
\section{Appendix}

\subsection{Proof of Lemma~\ref{lem_exist_bound}}
\label{proof_exist_bound}

\begin{proof}
From the real data, we have historical ${S}^{1:\tau}$ and  ${r}^{1:\tau}$: $\hat{S}^{1:\tau} = \{\hat{S}^1_1,...,\hat{S}^{\tau}_N \}$, $\hat{r}^{1:\tau} = \{\hat{r}^1_1,...,\hat{r}^{\tau}_N \}.$
Then $l_g^k = \min\{ \hat{S}^k_i/\hat{r}^k_i \}_{i \in \{1,...,N\}}$ is a global lower bound and $h_g^k = \max\{ \hat{S}^k_i/\hat{r}^k_i \}_{i \in \{1,...,N\}}$ is a global upper bound in time $k$ when there is no optimization on balancing. We call these two bounds global bounds. Let $l^k = l_g^k$, $h^k = h_g^k$ for all $k$, the supply-demand ratio after balancing should be contained in this range. We denote these global lower and upper bounds of the mobility supply-demand ratio as $\bar{l}^k$ and $\bar{h}^k$.

Consider the optimization problem~\eqref{opt:lem} that is similar to problem~\eqref{opt_minmax} except that problem~\eqref{opt:lem} uses the global bounds as lower and upper bounds in the quality constraints.
\begin{align}
	\begin{split} 
	\underset{\substack{X^{1:\tau}, Y^{1:\tau}, S^{1:\tau}, D^{1:\tau},\\ U^{1:\tau},  V^{2:\tau}, O^{2:\tau}, L^{2:\tau}}}{\text{min.}} \
		 &\underset{\{F_r\in \mathcal{F}_r,F_c\in \mathcal{F}_c\} }{\text{max.}}\ \mathbb{E}\left[J_D+\theta J_E \right]\\
		\text{s.t.}\ \text{\eqref{def_cons_mobility};}&\ \text{ \eqref{def_cons_charging};}\ \text{ \eqref{def_cons_moving};} \\
				r_i^k - \bar{l}^k S_i^k &- (D_i^k)^2 = 0,\\
        r_i^k - \bar{h}^k S_i^k &+ (U_i^k)^2 = 0,\\ 
        i \in \{1,\dots,N&\}, \ k \in \{1,\dots,\tau\}.
	\end{split}
	\label{opt:lem}
\end{align}
Suppose problem~\eqref{opt:lem} has feasible optimal solutions $S^{*1:\tau} = \{{S}^{*1}_1,...,{S}^{*\tau}_N \}$ and from the real data we have historical $\hat{r}^{1:\tau} = \{\hat{r}^1_1,...,\hat{r}^{\tau}_N \}$. Then we can compute $l^k$ and $h^k$ as following:
\begin{align}
	\begin{split}
       l^k = \max\{ l_g^k, \min\{ {S}^{*k}_i/\hat{r}^k_i \}_{i \in \{1,...,N\} }\},\\
        h^k = \min\{ h_g^k, \max\{ {S}^{*k}_i/\hat{r}^k_i \}_{i \in \{1,...,N\} }\}.
	\end{split}
\end{align}
It is obvious that $l^k \geq l_g^k$ and $h^k \leq h_g^k$, so the range $[l^k,g^k]$ is no wider than $[l_g^k, h_g^k]$. We find feasible lower and upper bounds that no loose than the global bounds.
\end{proof}

\subsection{Proof of Lemma~\ref{lem_convex}}
\label{proof_lem_convex}
\begin{proof}
For the generalized EV balancing optimization problem, let's first check the modified constraints and objective functions. The modified constraints \eqref{ge_bound}, \eqref{ge_trans} and \eqref{ge_con_charing} are still linear equality or inequality.
So if the objective function in problem~\eqref{ge_minmax} is convex over all $S_{e,i}^k,V_{e,i}^k,O_{e,i}^k,L_{e,i}^k$, the modified objective function is a convex problem of the decision variables since the composition of affine or linear operation preserves convexity [\cite{boyd2004convex}, Chapter 3.2.2]. Now we only need to check the minimization problem~\eqref{ge_obj_final} part in the generalized optimization problem is convex (since the maximization part is over the uncertain demand and supply parameters, not affected by the new formulation of $S_{e,i}^k,V_{e,i}^k,O_{e,i}^k,L_{e,i}^k$).
\begin{align}
	\begin{split} 
		\underset{\substack{X^{1:\tau}, Y^{1:\tau}, S^{1:\tau}, D^{1:\tau},\\ U^{1:\tau},  V^{2:\tau}, O^{2:\tau}, L^{2:\tau}}}{\text{min.}} \left[J_D^{\prime}+\theta J_E^{\prime} \right]
		\text{s.t.} \text{\eqref{ge_bound}; \eqref{ge_trans}; \eqref{ge_con_charing}; \eqref{def_cons_mobility_unfairness}. }
	\end{split}
	\label{ge_obj_final}
\end{align}

By the definition of $J_D^{\prime} (X^{1:\tau}_{1:E}, Y^{1:\tau}_{1:E})$ in \eqref{ge_obj_dist},
it is a linear function of $X^{k}_e$ and $Y^{k}_e$, hence also a convex function of
$X^{k}_e$
and $Y^{k}_e$,
$\forall k = 1,..., \tau;\ e = 1,..., E$
according to the definition of convex. And $J_E^{\prime}$ is a convex function of all decision variables for any fixed value of $c^k_e$, $\forall k = 1,..., \tau;\ e = 1,..., E$: this is because the function of power $\frac{1}{x^{a}}$ is convex on scalar $x>0$ when
$a>0$ [\cite{boyd2004convex}, Chapter 3.1.5],
so $\forall k = 1,..., \tau$, $J_{E}^{i,e,k} = \frac{c_{e,i}^k}{(T^k_i)^{a}}$ is a convex function on
$T^k_i>0$. For more detail, for $T_i^k = \sum_{e=1}^E T^k_{e,i}= \sum_{e=1}^E[\sum\limits_{j=1}^{N} y^{k}_{e,ji}-\sum\limits_{j=1}^{N} y^k_{e,ij}]$, with a matrix $B^i \in \mathbb{R}^{N\times N}$ that $B^i_{ji} = 1$,
$B_{ij} = -1,$ $Tr[B^iY^k_e] = T_{e,i}^k$. Then $J_{E}^{i,l,k} = 1/(\sum_eTr[B^iY^k_e])^{a}$ is a composition of convex function $1/x^{a}$ with an affine mapping: trace of the multiplication of metrics $B^i$ and $Y^k_e$. It's an operation that preserves convexity [\cite{boyd2004convex}, Chapter 3.2.2]. Finally, $\theta>0$,
$J_D^{\prime}+\theta J_E^{\prime}$, $J_E^{\prime} = \sum_{k=1}^{\tau} \sum_{e=1}^{L} \sum_{i \in \sigma} J_{E}^{i,e,k}$
are both weighted
sums of convex function, an operation that preserves convexity [\cite{boyd2004convex}, Chapter 3.2.1]. Hence, the minimization problem~\eqref{ge_obj_final} is a convex optimization problem.
\end{proof}
\vspace*{-10pt}
\subsection{Proof of Lemma~\ref{lem_alg}}
\label{proof_lem_alg}
\begin{proof}
Without loss of generality, we prove the case for constraint parameter $\gamma_\alpha$. Then according to the definition of quantiles $q_{\eta/2}^{\gamma}$ and $q_{1-\eta/2}^{\gamma}$ we have:
$1-\eta = P(q_{\eta/2}^{\gamma} \leq \frac{\tilde{\gamma}_{\alpha} - \hat{\gamma}_{\alpha}}{{s}_{\gamma}}  \leq q_{1-\eta/2}^{\gamma}).$
Under the assumption that the distribution of $\frac{\tilde{\gamma}_{\alpha} - \hat{\gamma}_{\alpha}}{{s}_{\gamma}}$ is close to the distribution of $\frac{\hat{\gamma}_{\alpha} - {\gamma}_{\alpha}}{{s}_{\gamma}}$ [\cite{arnab2017survey}, Chapter 18.6.1.1.2], we have: $P(q_{\eta/2}^{\gamma} \leq \frac{\tilde{\gamma}_{\alpha} - \hat{\gamma}_{\alpha}}{{s}_{\gamma}}  \leq q_{1-\eta/2}^{\gamma}) = P(q_{\eta/2}^{\gamma} \leq \frac{ \hat{\gamma}_{\alpha} - \gamma_{\alpha}}{s_{\gamma}}  \leq q_{1-\eta/2}^{\gamma}) = P(\hat{\gamma}_{\alpha}-s_{\gamma}q_{1-\eta/2}^{\gamma} \leq \gamma_{\alpha}  \leq \hat{\gamma}_{\alpha} -s_{\gamma}q_{\eta/2}^{\gamma}) = 1-\eta$.
So the probability that $\gamma_{\alpha}$ is  between $\hat{\gamma}_{\alpha} - s_{\gamma}q_{\eta/2}^{\gamma}$ and $\hat{\gamma}_{\alpha1} - s_{\gamma}q_{1-\eta/2}^{\gamma}$ equals to $1-\eta$.
\end{proof}

\vspace*{-10pt}
\subsection{Proof of Theorem~\ref{thm:theory1}}
\label{proof_theorem1}
\begin{proof}
We notice that the uncertainty parameters are involved in both objective functions and constraints. We first use $\underset{F_r\in \mathcal{F}_r}{\text{max.}} \mathbb{E}(r_i^k)$ substitute $r_i^k$, $\underset{F_c\in \mathcal{F}_c}{\text{min.}} \mathbb{E}(c_i^k)$ substitute $c_i^k$ in constraints \eqref{def_cons_mobility} and \eqref{def_cons_mobility_unfairness}. It's valid for minimizing the worst case: any uncertain values of $r_i^k$ and $c_i^k$ should meet the relationship with other decision variables in constraints \eqref{def_cons_mobility} and \eqref{def_cons_mobility_unfairness}. No matter which probability distribution is selected as the specific distribution to attain the worst case, the simplest worst case for single value of $r_i^k$ and $c_i^k$ is the case that the demand is really large that attains the maximal possible demand value while the supply is very small that attains the minimal possible supply value. And $\mathbb{E}(r_i^k) (\mathbb{E}(c_i^k))$ is the probability-weighted average of all its possible values. Then we transfer all constraints into functional formats as below:
\vspace{-5pt}
\begin{align}
\begin{split}
	    f_{D_i^k} &= \underset{F_r\in \mathcal{F}_r}{\text{max.}} \mathbb{E}(r_i^k) - l_i^k S_i^k - (D_i^k)^2 = 0,\\
        f_{U_i^k} &= \underset{F_r\in \mathcal{F}_r}{\text{max.}} \mathbb{E}(r_i^k) - h_i^k S_i^k + (U_i^k)^2 = 0,\\ 
	    f_{S_i^k} &= -S^k_i + X_i^k + V^k_i = 0,\quad k=1,\dots,\tau, \\
		f_{V_i^{k+1}} &= -V^{k+1}_i + \sum\limits_{j=1}^{N} P^k_{vji}S^k_j+ \sum\limits_{j=1}^{N} Q^k_{vji}O^k_j \\
		& \quad - \underset{F_c\in \mathcal{F}_c}{\text{max.}} \mathbb{E}(-c_i^k) = 0,\\ 
		f_{O_i^{k+1}} &= -O^{k+1}_i + \sum\limits_{j=1}^{N}P^k_{oji}S^k_j+ \sum\limits_{j=1}^{N} Q^k_{oji}O^k_j = 0,,\\
	    f_{L_i^{k+1}} &= - L^{k+1}_i + Y_i^k + \sum\limits_{j=1}^{N} P_{lji}^k S^k_{j} = 0, 
	     L_i^k > 0, \\\quad k&=1,\dots,\tau-1,\quad
	     S_i^k > 0, \quad k=1,\dots,\tau.
	\end{split}
\label{con_functional}
\end{align}
	Let $f_D = [f_{D_1^1},f_{D_1^2},...,f_{D_1^\tau},...,f_{D_N^\tau}]^T \in \mathbb{R}^{N\tau}$ be a constraint function vector, for $i = 1, \dots,N, k = 1,\dots,\tau-1$, and $f_U,f_S,f_V,f_O,f_L$ have the same definition but for computational convenient, if one's dimension is less than $N\tau$, we add 0 in corresponding missing positions to complete its dimension. For the primal maximization problem 
	\begin{align*}
    \underset{F_r\in \mathcal{F}_r,F_c\in \mathcal{F}_c }{\text{max.}}\ &\mathbb{E}\left[J_D+\theta J_E \right],\quad
		\text{s.t.   } \eqref{con_functional}
	\end{align*}
The primal objective function only contains uncertainty parameter $c$ and is concave over $c$, because it's a linear function of $c$ given other decision variables. The constraints are also all linear in $r$ and $c$. For this primal problem, strong duality and Slater’s theorem hold according to~\cite{boyd2004convex}. Then the Lagrange dual problem \eqref{dual} can obtain its best upper bound: 
	\begin{align}
	\begin{split}
	   & \underset{\lambda \succeq 0}{\text{min}}\
		\underset{F_r\in \mathcal{F}_r,F_c\in \mathcal{F}_c }{\text{max}}\ J_{dual},\\
	    J_{dual} &= \mathbb{E}\left[J_D+\theta J_E \right] - (\lambda_U^T f_U + \lambda_S^T f_S + \lambda_V^T f_V \\
	    &+ \lambda_O^T f_O + \lambda_L^T f_L + \lambda_s^T S^{1:\tau} + \lambda_l^T L^{2:\tau})
	\end{split} 
	\label{dual}
	\end{align}
Here $\lambda_U^T, \lambda_S^T, \lambda_V^T, \lambda_O^T, \lambda_L^T, \lambda_s^T, \lambda_l^T$ are the corresponding Lagrange multipliers and $\lambda$ is defined as a concatenated vector combined by all these Lagrange multipliers. We have $\frac{c_i^k}{(T^k_i)^a} \geqslant 0$ and $c_i^k \geqslant 0$ by the definitions of $J_E$ in~\eqref{def_charging_unfairness}, then for any vector $Z\in \mathbb{R}^{N\tau}$, $Z=[z^1_1, z^1_2,\dots, z^{\tau}_1, z^{\tau}_2, \dots,z^{\tau}_{N\tau}]^T$ that satisfies $0<\frac{1}{(T^k_i)^{a}} \leqslant z_{i}^k$, we also have\\
	\centerline{$
		0 \leqslant \sum_{k=1}^{\tau} \sum\limits_{i=1}^{N} \frac{c^k_i}{(T^k_i)^a} \leqslant Z^T c,
		$}
and the second inequality strictly holds when all $\frac{c^k_i}{(T^k_i)^{a}}=z_i^k$, for $i=1,\dots,N$, $k=1,\dots, \tau$. The constraints of problem~\eqref{dual} are independent of $c$, hence, for any $c$, the minmax problem \eqref{dual} is equivalent to 
	\begin{align}
	    \begin{split}
	        \underset{\lambda \succeq 0}{\text{min}}\
		    \underset{F_r\in \mathcal{F}_r,F_c\in \mathcal{F}_c }{\text{max}}\ &J_{dual}^\prime
		    \text{s.t.} \frac{1}{(T_i^k)^a} \leqslant z_i^k, \quad Z \in \mathbb{R}^{N\tau},
	\end{split}\label{dual_2}
	\end{align}
	where $J_{dual}^\prime = \mathbb{E}\left[J_D+\theta Z^T c \right] - (\lambda_U^T f_U + \lambda_S^T f_S + \lambda_V^T f_V 
	    + \lambda_O^T f_O + \lambda_L^T f_L + \lambda_s^T S^{1:\tau} + \lambda_l^T L^{2:\tau})$.
	We separate $J_{dual}^\prime$ into three parts: $H_r =  -(\lambda_U^T + \lambda_D^T) r, 
	    H_c =  \theta J_E - \lambda_V^T c, H_o = J_{dual}^\prime - \mathbb{E}[H_c + H_r],$ only $H_r$ contains all $r$, $H_c$ contains all $c$. $H_o$ can be put as a deterministic value given other decision variables. So we have the following maximization problem
	\begin{align}
		\underset{r \sim F_r, c \sim F_c, F_r\in \mathcal{F}_r,F_c\in \mathcal{F}_c}{\text{max}} \mathbb{E} [H_r + H_c].
		\label{obj_max_all}
	\end{align}
	Since $r$ and $c$ are independent, problem \eqref{obj_max_all} is equivalent to the separated maximization problem
	\begin{align}
		\underset{r \sim F_r, F_r\in \mathcal{F}_r}{\text{max}} \mathbb{E} [H_r] + \underset{c \sim F_c, F_c\in \mathcal{F}_c}{\text{max}} \mathbb{E}[H_c].
		\label{obj_max_sep}
	\end{align}
	Problem \eqref{obj_max_sep} satisfies the conditions of Lemma 1 in~\cite{delage2010distributionally}, and the maximum expectation value of $H_r + H_c$ for any possible $r \sim F_r ,c \sim F_c$ where $F_r \in \mathcal{F}_r ,F_c \in \mathcal{F}_c$ equals the optimal value of the problem
		\begin{align}
		\begin{split}
			\min_{\substack{Q_r,q_r,v_r,t_r;\\Q_c,q_c,v_c,t_c}}\quad &v_r + t_r + v_c + t_c\\
			\text{s.t.}
			\quad & v_r\geqslant H_r -r^T Q_r r -r^T q_r, Q_r, Q_c\succeq 0\\
			&t_r \geqslant (\hat\gamma_{r} \hat{\Sigma}_r+\hat{r}\hat{r}^T)\cdot Q_r + \hat{r}^T q_r
			\\&\quad\quad+\sqrt{\hat\omega_{r}}
			\|\hat{\Sigma}_r^{1/2} (q_r+2Q_r\hat{r})\|_2 ,\\
			\quad & v_c\geqslant H_c -c^T Q_c c -c^T q_c,\\
			&t_c \geqslant (\hat\gamma_{c} \hat{\Sigma}_c+\hat{c}\hat{c}^T)\cdot Q_c + \hat{c}^T q_c \\&\quad\quad+\sqrt{\hat\omega_{c}} \|\hat{\Sigma}_c^{1/2} (q_c+2Q_c\hat{c})\|_2.
		\end{split}
		\label{step1}
	\end{align}
	Note that the first and third constraints about $v_r$ and $v_c$ is equivalent to $v_r \geqslant f_r(r^*)$ and $v_c \geqslant f_c(c^*)$ where $f_r(r^*) ( f_c(c^*) )$ is the optimal value of the following problem 
	    \begin{align}
	        \begin{split}
	            \max_{r} \quad & H_r - r^T Q_r r -r^T q_r\quad
	                 \text{s.t.} \quad r \geqslant 0.\\
	            \max_{c} \quad & H_c - c^T Q_c c -c^T q_c\quad
	                 \text{s.t.} \quad c \geqslant 0.
	        \end{split}
	        \label{obj_cons}
	    \end{align}{}
	Since $Q_r$ and $Q_c$ are positive semi-defined, $H_r (H_c)$ is a linear function over $r (c)$, problem \eqref{obj_cons} is convex. Solving this problem by taking partial derivative over $r (c)$ without constraints, we have:
	    \begin{align}
	        \begin{split}
	            v_r & \geqslant\frac{1}{4}(q_r + \lambda_U + \lambda_D)^T Q_r^{-1} (q_r + \lambda_U + \lambda_D)\\
	            v_c & \geqslant\frac{1}{4}(q_c + \lambda_V - Z)^T Q_c^{-1} (q_r + \lambda_V - Z)
	        \end{split}{}
	    \end{align}{}
	By Schur complement, the above constraints are
	    \begin{align}
	        \begin{split}
	            &\quad  \begin{bmatrix}v_r & \frac{1}{2}(q_r+\lambda_U + \lambda_D)^T\\ \frac{1}{2}(q_r+\lambda_U + \lambda_D) & Q_r 
			\end{bmatrix} \succeq 0,\\ \\
			&\quad \quad \begin{bmatrix}v_c & \frac{1}{2}(q_c+\lambda_V - Z)^T\\ \frac{1}{2}(q_c+\lambda_V - Z) & Q_c 
			\end{bmatrix} \succeq 0,\\
	        \end{split}{}
	    \end{align}{}
But these constraints are under the conditions of no constraints in problem \eqref{obj_cons}. We still have another two constraints for $v_r (v_c)$ that $v_r \geqslant f_r(r = 0) = 0$ and $v_c \geqslant 0$. We then use the fact that min-min operations can be performed jointly and combine all constraints to reformulate problem \eqref{dual} as \eqref{thm:theory1}.
\end{proof}

\definecolor{bln_blue}{HTML}{AED6F1}
\definecolor{bln_red}{HTML}{F5B7B1}
\definecolor{bln_green}{HTML}{D4EFDF}
\definecolor{bln_magenta}{HTML}{9B59B6}

\vspace*{-10pt}
\subsection{Non-robust Method Used in Experiment}
\input{sup_009_non_robust}

%% file: sup_009_non_robust.tex
The non-robust method treats the predictions of demand and supply as deterministic vales. Therefore, the non-robust method is solving a minimization problem \eqref{non_robust_prob} instead of a minimax problem.
\begin{align}
	\begin{split}
		\underset{\substack{X^{1:\tau}, Y^{1:\tau}, S^{1:\tau}, D^{1:\tau},\\ U^{1:\tau},  V^{2:\tau}, O^{2:\tau}, L^{2:\tau}}}{\text{min.}} \mathbb{E} \left[ J_D + \theta J_E \right]
\text{s.t. } \text{\eqref{def_cons_mobility}, \eqref{def_cons_charging}, \eqref{def_cons_moving}, \eqref{def_cons_mobility_unfairness}, }
	\end{split}
	\label{non_robust_prob}
\end{align}
In the minimization problem \eqref{non_robust_prob}, $J_D$ is defined in Eq.~\eqref{obj_dist} and $J_E$ is defined in Eq.~\eqref{def_charging_unfairness}, $\theta$ is a positive coefficient which is chosen as the same value of the weight in the distributionally robust optimization EV balancing problem formulation (the robust minimax problem), i.e. Equation \eqref{opt_final}. We have discussed or proved that the objective functions $J_D$ and $J_E$ are both convex over all variables in the text. And all constraints \eqref{def_cons_mobility}, \eqref{def_cons_charging}, \eqref{def_cons_moving},  and \eqref{def_cons_mobility_unfairness} are convex (linear or quadratic). Therefore, this non-robust minimization problem can be solved through any convex optimization solver. This non-robust minimization problem is in a receding horizon control paradigm as well. So we are able to use the same parameters and real-sensing data used in our DRO method. The predictions of demand $d$ and supply $c$ are obtained from the same well-trained ARIMA model used in our DRO method.